\documentclass{amsart}
\usepackage{amssymb, amsmath, color}

\newtheorem{theorem}{Theorem}
\newtheorem{lemma}[theorem]{Lemma}
\newtheorem{corollary}[theorem]{Corollary}
\newtheorem{remark}[theorem]{Remark}

\newcommand{\id}{\operatorname{Id}} 
\newcommand{\man}{M} 
\newcommand{\affcon}{D} 
\newcommand{\D}{\mathcal{D}} 
\newcommand{\affman}{\Sigma} 
\newcommand{\J}{\mathfrak{J}} 
\newcommand{\G}{\mathfrak{G}} 

\begin{document}

\title[Geometric
properties of generalized symmetric spaces]
{Geometric
properties of generalized symmetric spaces}
\author{\quad E. Calvi\~{n}o-Louzao, E. Garc\'{\i}a-R\'{\i}o, M.E. V\'{a}zquez-Abal, R. V\'{a}zquez-Lorenzo}
\address{Department of Geometry and Topology, Faculty of Mathematics,
University of Santiago de Compostela, 15782 Santiago de Compostela, Spain}
\email{estebcl@edu.xunta.es, eduardo.garcia.rio@usc.es, elena.vazquez.abal@usc.es,
ravazlor@edu.xunta.es}
\thanks{Supported by project MTM2009-07756 (Spain)}
\subjclass{53C50, 53B30}
\date{}
\keywords{Homogeneous spaces, Generalized symmetric spaces, K\"{a}hler and para-K\"{a}hler manifolds, Riemannian extensions.}

\begin{abstract}
It is shown that four-dimensional generalized symmetric spaces can be naturally equipped with some additional structures defined by means of their curvature operators. As an application, those structures are used to characterize generalized symmetric spaces.
\end{abstract}

\maketitle

\section*{Introduction}

The study of the curvature is a central problem in pseudo-Riemannian geometry. It is well-known that the existence of some additional structure (K\"{a}hler, Sasakian, quaternionic K\"{a}hler, etc.) strongly influences the properties of the curvature. The K\"ahler form of any four-dimensional K\"ahler manifold (or more generally, of  some classes of Hermitian or almost K\"ahler manifolds) is an eigenvector of the self-dual Weyl curvature operator. A less explored problem is the converse one, which focuses on the existence of some conditions on the curvature which allow the  construction of  additional structures on the manifold, aimed to characterize   the underlying geometry in terms of the properties of the new structures. The generalized Goldberg-Sachs Theorem is a classical example (see \cite{Apost} and the references therein), showing that any Einstein four-dimensional manifold with a special self-dual Weyl curvature operator is naturally endowed with a complex structure resulting in a locally conformally K\"ahler manifold.

The purpose of this paper is to contribute to the above general problem by showing some characterization results of a specially nice class of pseudo-Riemannian manifolds, namely the four-dimensional generalized symmetric spaces.
It is a well-known fact that any $3$-symmetric space admits an almost Hermitian structure. In this paper we investigate the existence of some additional structures on generalized symmetric spaces. Our approach is based on the study of the curvature operator and the possibility of constructing some structures induced by certain $2$-forms which exhibit a distinguished behaviour with respect to the curvature operator, inspired  by the Goldberg-Sachs Theorem.

We show that any four-dimensional generalized symmetric space is either conformally symmetric (i.e., the Weyl curvature tensor is parallel) or otherwise the self-dual and the anti-self-dual Weyl curvature operators have a distinguished eigenvalue of multiplicity one. Local sections of the corresponding one-dimensional eigenbundles define a structure of symplectic pair and we will show that the different classes of four-dimensional generalized symmetric spaces are characterized by the properties of these structures.

{Generalized} symmetric spaces have been extensively investigated from different points of view during the last years. Curvature properties of four-dimensional generalized symmetric spaces have been firstly considered in \cite{Calvaruso-Leo08} and later used in \cite{CZ} and \cite{LV} to investigate their underlying pseudo-Riemannian structure and parallel hypersurfaces in those model spaces. The existence of harmonic vector fields and Ricci solitons has been studied in \cite{BO} and \cite{Calvaruso2}. Finally, it  is worth mentioning that some additional complex and para-complex structures on four-dimensional generalized symmetric spaces have been constructed in \cite{Calvaruso}.

The paper is organized as follows. We recall some basic material in Section \ref{section-1}, focusing on the decomposition of the curvature operator and the existence of almost Hermitian and almost para-Hermitian structures in dimension four. Walker structures appear in a natural way and we recall some general  facts in Section \ref{section-Walker} to be used in subsequent sections.
The geometry of four-dimensional generalized symmetric spaces is treated in Section \ref{esg-dim-4}, where their curvature is investigated and some additional structures are constructed. Finally,  we use  these structures in Section \ref{section-characterizations} to characterize the different types of four-dimensional generalized symmetric spaces. Although we have tried to avoid calculations as much as possible, an Appendix is included with a detailed expression of the Levi-Civita connection and the curvature tensor for the three classes of four-dimensional generalized symmetric spaces.

\section{Preliminaries}\label{section-1}

Throughout this paper, $(M,g)$ is a pseudo-Riemannian manifold and $R$ denotes the
curvature tensor taken with the sign convention $R(X,Y)=\nabla_{[X,Y]}-[\nabla_X,\nabla_Y]$,
$\nabla$ being the Levi-Civita connection. The Ricci tensor, $\rho$, and the corresponding Ricci
operator, $Ric$, are given by $\rho(X,Y)=g( Ric\,X,Y)$ $=$ $\operatorname{trace}\{Z \mapsto
R(X,Z)Y\}$. The scalar curvature $\tau$ is defined as the trace of the Ricci tensor. Finally, we denote by  $\circ$   the symmetric product given by
$\xi_1\circ\xi_2=\frac{1}{2}(\xi_1\otimes\xi_2+\xi_2\otimes\xi_1)$.

The curvature of any pseudo-Riemannian manifold $(M,g)$, as an endomorphism of the space of 2-forms $\Lambda^2=\Lambda^2(M)$, naturally decomposes under the action of the orthogonal group as
$R=\frac{\tau}{n(n-1)}\operatorname{Id}_{\Lambda^2} + \rho_0+W$, where $n=\dim M$, $\rho_0=\rho-\frac{\tau}{n}g$  is the trace-free Ricci tensor and $W$ is the Weyl curvature tensor.

\subsection{Four-dimensional geometry}

In dimension  four the  Hodge star operator allows to define the self-dual and anti-self-dual Weyl curvature operators. Let $\man$ be a manifold and $g$ a metric of arbitrary signature. Let $\{e_1,e_2,e_3,e_4\}$ be an orthonormal basis  and let $\{e^1,e^2,e^3,e^4\}$ be the associated dual basis.
Considering the space of $2$-forms
\[
    \Lambda^2(M)=\operatorname{Span}  \{e^i\wedge e^j: \,
i,j\in\{1,2,3,4\},\, i<j\},
\]
the  Hodge  star operator $\star$ acts on
$\Lambda^2(M)$ as
\[
e^i\wedge e^j\wedge \star(\,e^k\wedge e^\ell)=(\delta_k^i
\delta_\ell^j-\delta_\ell^i\delta_k^j)\,\varepsilon_i\varepsilon_j\, e^1\wedge
e^2 \wedge e^3 \wedge e^4,
\]
where $\varepsilon_i=g( e_i,e_i)$ and  $\delta_i^j$ denotes the
Kronecker symbol. The properties of  the Hodge  star operator depend on the signature of $g$  and it defines a product structure in the Riemannian and  neutral signature settings (i.e.,
$\star^2=\id_{\Lambda^2(M)}$), inducing a splitting
$\Lambda^2(M)=\Lambda^+(M)\oplus\Lambda^-(M)$, where $\Lambda^+=\Lambda^+(M)$
and $\Lambda^-=\Lambda^-(M)$ denote the spaces of  self-dual and anti-self-dual $2$-forms,  respectively:
\[
\Lambda^+=\{\alpha\in \Lambda^2: \star \alpha=\alpha\}, \qquad
\Lambda^-=\{\alpha\in \Lambda^2: \star \alpha=-\alpha\}.
\]
An orthonormal basis for the self-dual and anti-self-dual space is given by
\[
    \Lambda^\pm=\operatorname{Span}\{E_1^\pm,E_2^\pm,E_3^\pm\},
\]
 where
\begin{equation}\label{PRELIMbasedosformas}
\begin{array}{l}
E_1^\pm=\, ( e^1\wedge e^2\pm \varepsilon_3\varepsilon_4 e^3\wedge e^4)/\sqrt{2},\\
\noalign{\medskip}
E_2^\pm=\, ( e^1\wedge e^3\mp \varepsilon_2\varepsilon_4 e^2\wedge e^4)/\sqrt{2},\\
\noalign{\medskip} E_3^\pm=\, ( e^1\wedge e^4\pm
\varepsilon_2\varepsilon_3 e^2\wedge e^3)/\sqrt{2}.
\end{array}
\end{equation}
The induced metric on $\Lambda^2$   is given by
\[
\langle\langle x\wedge y,z\wedge w\rangle\rangle=g(
x,z)g( y,w) -g( y,z)g( x,w),
\]
which is Riemannian if $g$ is  positive definite or of signature $(----++)$ if $g$ is of neutral signature
$(2,2)$. Moreover, in the latter case, the restrictions of the metric to the subspaces $\Lambda^\pm$ are of Lorentzian signature $(--+)$ and  $\{E_1^\pm,E_2^\pm,E_3^\pm\}$ is an orthonormal basis where
$E_1^\pm$ are spacelike and  $E_i^\pm$ are timelike for $i=2,3$.

Since $\Lambda^2$ splits  under the action of the  Hodge star operator so does the curvature acting on the space of 2-forms as
$$
R=\frac{\tau}{12}\id_{\Lambda^2(M)}+\rho_0+W^++W^-
$$
where $W^\pm=\frac{1}{2}(W\pm \star W)$ denote the self-dual and the  anti-self-dual Weyl curvature operators.

\subsection{Almost Hermitian and almost para-Hermitian structures}

An \emph{almost Hermitian manifold} is a pseudo-Riemannian manifold $(M,g)$ equipped with an orthogonal almost complex structure (i.e.,  a $(1,1)$-tensor field $J$ satisfying $J^2=-\id$ and $J^*g=g$). Any almost Hermitian structure $(g,J)$ gives rise to a 2-form $\Omega(X,Y)=g(JX,Y)$ inducing an orientation which agrees with the one defined by the almost complex structure in the Riemannian case. However, if the metric $g$ is of signature $(--++)$ then both $J$ and $\Omega$ induce opposite orientations. Hence, orienting $M$ by the almost complex structure, the K\"ahler form $\Omega$ is  self-dual in the positive definite case and anti-self-dual in the neutral signature case. A straightforward calculation shows that $\Omega$ is a spacelike section of  $\Lambda^+$ or $\Lambda^-$ (depending on the signature of $g$) and conversely, any such spacelike section defines an almost Hermitian structure on $M$. Moreover, $(M,g,J)$ is said to be \emph{almost K\"ahler} if the 2-form $\Omega$ is closed and $(M,g,J)$ is said to be \emph{K\"ahler} if, in addition, the almost complex structure is integrable (or, equivalently, the K\"ahler form $\Omega$ is parallel).

A \emph{para-K\"ahler manifold} is a symplectic manifold
admitting two transversal Lagrangian foliations (see \cite{Ivanov-Zamkovoy05} and the references therein). Such a structure induces a decomposition of the tangent bundle $TM$ into the Whitney sum of Lagrangian subbundles $L^+$ and $L^-$, that is, $TM=L^+\oplus L^-$. By
generalizing this definition, an \emph{almost para-Hermitian
manifold} {is} an almost symplectic manifold $(M,\Omega)$ whose
tangent bundle splits into the Whitney sum of Lagrangian
subbundles. This definition implies that the $(1,1)$-tensor field
$\J$ defined by $\J=\pi_{L^+}-\pi_{L^-}$ is an \emph{almost para-complex
structure} (i.e.,  $\J^2=\id$ on $M$) such that $\Omega(\J X,\J Y)$ $=$
$-\Omega(X,Y)$ for all $X,Y\in\Gamma TM$, where $\pi_{L^\pm}$ are the projections of $TM$ onto $L^\pm$. The 2-form $\Omega$ induces a non-degenerate
$(0,2)$-tensor field $g$ on $M$ defined by $g(X,Y)$ $=$
$\Omega(\J X,Y)$ such that $g(\J X,\J Y)=-g(X,Y)$ for all $X,Y\in\Gamma TM$. It follows now that the
$2$-form $\Omega$ defines a timelike section of $\Lambda^+$, and
conversely, any timelike section of $\Lambda^+$ defines an
almost para-Hermitian structure.
Finally, an almost para-Hermitian structure $(g,\J)$ is said to be
\emph{almost para-K\"ahler} if the associated 2-form $\Omega$ is closed.

Throughout  this paper   an almost para-Hermitian structure is said to be  \emph{opposite} if the para-K\"ahler form $\Omega\in\Gamma\Lambda^-$ and an almost Hermitian structure is said to be \emph{opposite} if the K\"ahler form $\Omega\in\Gamma\Lambda^-$ ($\Omega\in\Gamma\Lambda^+$ in the neutral signature case). The existence of a (para)-K\"ahler  structure strongly influences the curvature of the manifold since, in any   case, the (para)-K\"ahler form is a distinguished eigenvector of the self-dual Weyl curvature operator, which is diagonalizable of the form
$W^+=\operatorname{diag}\frac{\tau}{12}[2,-1,-1]$. We emphasize that the K\"{a}hler form is anti-self-dual in the almost Hermitian neutral case of signature $(2,2)$ so the above holds true for the anti-self-dual curvature operator $W^-$ in this case.

A \emph{symplectic pair} on a four-dimensional manifold $M$ is a pair of non-trivial closed $2$-forms $(\omega,\eta)$  of constant and complementary ranks, for which $\omega$ (resp. $\eta$)
restricts to the leaves of the kernel foliation of $\eta$ (resp. $\omega$) as a symplectic form
\cite{BK}. Then $\Omega_\pm=\omega\pm\eta$ are symplectic forms on $M$ compatible with different orientations. Indeed, a symplectic pair can be equivalently defined by two symplectic forms
$(\Omega_+,\Omega_-)$ satisfying
\begin{equation}\label{eq:intro}
\Omega_+\wedge\Omega_-=0,
\qquad
\Omega_+\wedge\Omega_+=-\Omega_-\wedge\Omega_-\,.
\end{equation}

If $(\Omega_+,\Omega_-)$ is a symplectic pair, then the kernels of $\Omega_+\pm\Omega_-$ define complementary foliations with minimal leaves, and conversely, any symplectic pair on a four-dimensional manifold is given by a pair of
two-dimensional oriented complementary minimal foliations on $M$
\cite{BK, BCGHM}.

\subsection{Walker structures and Riemannian extensions}\label{section-Walker}

It is immediate to recognize that the $\pm1$-eigenspaces of a para-complex structure $\J$ define isotropic subbundles $L^\pm$ of $TM$ which are parallel (i.e.,
$\nabla L^\pm\subset L^\pm$) in the para-K\"ahler setting. Hence any para-K\"ahler manifold
has an underlying Walker structure, i. e. it admits a parallel distribution where the metric tensor degenerates  (see \cite{Brozos-Garcia-Gilkey-Nikevic-Vazquez09} for more information on Walker manifolds).

A special family of Walker structures is the one defined by the so-called Riemannian  extensions
(we refer to  \cite{Patterson-Walker52} for further details  and also to  \cite{CGRGVL09, KSek} for a modern exposition).
Since Riemannian  extensions will also play a distinguished role in our study, next we recall their construction. We restrict to the four-dimensional setting, which is the purpose of this paper.
Let $(\affman,\affcon)$ be an affine surface, denote by $T^\ast \affman$ its cotangent bundle, and let $\pi:T^\ast \affman\rightarrow\affman$ be the canonical projection. For each symmetric $(0,2)$-tensor field $\Phi$ on $\affman$, $T^\ast \affman$ can be naturally equipped with a pseudo-Riemannian metric $g_{\affcon,\Phi}$ of signature $(--++)$, the \emph{Riemannian  extension} of $\affcon$ and $\Phi$, which is determined by\[
    g_{\affcon,\Phi}(X^C,Y^C)=-\iota(\affcon_XY + \affcon_YX)+(\pi^*\Phi)(X^C,Y^C),
\]
where  $X^C$, $Y^C$ are the complete lifts to $T^\ast
\affman$ of vector fields $X$, $Y$ on $\affman$ and $\iota$ is the evaluation map. In the induced coordinate system
$(x^i,x_{i'})$ on $T^\ast \affman$, the Riemannian extension is given by
\begin{equation}\label{eq:Riemannianextension}
    g_{\affcon,\Phi}=\left(
    \begin{array}{cc}
        -2x_{k'}\Gamma_{ij}^k +\Phi_{ij}& \id_2
        \\[0.1in]
        \id_2  & 0
        \end{array}
    \right),
\end{equation}
with respect to the local basis of coordinate vector fields, where $\Gamma_{ij}^k$ are the Christoffel symbols of the affine connection $\affcon$ with respect to the coordinates $(x^i)$ on $\affman$ and $\Phi=\Phi_{ij}dx^i\otimes dx^j$. Moreover the distribution  $\D=\operatorname{ker}\,\pi_*$ is null and parallel.

Riemannian  extensions relate the affine geometry of $(\affman,\affcon)$ with the pseudo-Riemannian geometry of $(T^*\affman, g_{\affcon,\Phi})$ and moreover, they are the underlying structure of many Walker manifolds. For instance, self-dual Walker manifolds as well as Ricci flat Walker manifolds are Riemannian  extensions up to some modifications (see \cite{Brozos-Garcia-Gilkey-Nikevic-Vazquez09,Garcia-Gilkey-Nikevic-Vazquez13} for details).

\section{Four-dimensional generalized symmetric spaces}\label{esg-dim-4}

Symmetric spaces are geometrically characterized by the fact that the (local) geodesic symmetries are isometries. This classical result motivated the study of Riemannian manifolds whose geodesic symmetries satisfy some weaker condition aimed to characterize some generalizations of locally symmetric spaces. The study of D'Atri spaces \cite{DN} is just a significant example, where the geodesic symmetries are assumed to be volume-preserving. Within this general problem, other kinds of geodesic reflections (with respect to higher-dimensional submanifolds) were also investigated \cite{Chen-Vanhecke, Chen-Vanhecke2, Willmore-Vanhecke} as well as other kinds of geodesic transformations generalizing the similarities and inversions of the Euclidean space (see
\cite{DAtri, EGR-Vanhecke}).

Generalized symmetric spaces appear in a natural way when considering local transformations in a neighbourhood of a fixed point which are not necessarily involutive ($s_p^2=\id$), but satisfying $s_p^k=\id$ for some constant $k\geq 3$. Recall that a \emph{$s$-structure} in a pseudo-Riemannian manifold $(\man,g)$ is a family $\{s_p, p\in \man\}$ of isometries of the manifold (which are called
 \emph{symmetries in $p$}) such that $p$ is an  isolated fixed point.  A $s$-structure is said to be \emph{regular} if the application $(p,q)\in\man\times\man\to s_p(q)$
is smooth and  for each pair of points
$p,q \in \man$ one has $s_p \circ s_q=s_{\overline{p}} \circ s_p$, where $\overline{p}=s_p(q)$.
Moreover, a $s$-structure $\{s_p\}$ is said to be of \emph{order $k$}
$(k\geq 2)$ if for all $p\in \man$, $s_p^k=\id_\man$, and $k$ is the minimum of all natural numbers satisfying this property. A $s$-structure is of \emph{infinity order} if there is no natural number $k$ such that $s_p^k=\id_\man$. Examples of $s$-structures of order $3$ and $4$ are naturally related to the existence of certain almost Hermitian structures (see, for example,  \cite{Gray, Ni-Va}).

A \emph{generalized symmetric space} is a pseudo-Riemannian manifold $(\man,g)$ which admits at least a regular $s$-structure. Kowalski showed that all generalized symmetric spaces are necessarily homogeneous and classify them in dimension $\leq 5$  \cite{Kowalski80}.
While the only three-dimensional (Riemannian or Lorentzian) generalized symmetric space is the Lie group $\operatorname{sol}_3$, the situation is richer in dimension four. Despite of the fact that there is a single Riemannian four-dimensional generalized symmetric space (which is of order $3$), in the neutral signature setting there exist generalized symmetric spaces which are of order $3$ or infinity. (Four-dimensional examples do not exist in the Lorentzian signature).
Besides their own interest, it is worth mentioning that generalized symmetric spaces are the underlying structure of some natural geometrical conditions (see, for example, \cite{Apostolov-Armstrong-Draghici02, Fino05} and the references therein).

The classification of pseudo-Riemannian four-dimensional generalized symmetric spaces was given by $\check{C}$ern\'y  and Kowalski in \cite{Cerny-Kowalski82}, where they obtained four types of generalized symmetric spaces, which actually reduce to three as Remark \ref{re:type C} shows. In what follows we consider the three non-symmetric different classes and investigate their geometry,  constructing some additional structures which, in some cases completely characterize them.
Some of these structures  were already constructed by Calvaruso \cite{Calvaruso} in a purely algebraic way. It is worth emphasizing that our construction strongly  depends on curvature properties. Most of the calculations, although long and tedious, are rather straightforward and therefore we have kept them at minimum aimed to make the paper more clear and readable.

\smallskip
\subsection{Almost K\"{a}hler and opposite K\"{a}hler structure of generalized symmetric spaces of Type I}

Generalized symmetric spaces of Type I are of order $3$, and the corresponding pseudo-Riemannian metric may be positive definite or of neutral signature. The model space
$(\man,g)$ corresponds to  $\mathbb{R}^4$ with coordinates $(x,y,u,v)$ and metric
\begin{equation}\label{GENERtype A metric}
\begin{array}{rcl}
g\!\!&\!\!=\!\!&\!\pm[(\sqrt{1+x^2+y^2}-x)du\!\circ\! du +(\sqrt{1+x^2+y^2}+x)dv\!\circ\! dv-2y
du\!\circ\! dv]\\
\noalign{\medskip}
 &&+\lambda[(1+y^2)dx\!\circ\! dx +(1+x^2)dy\!\circ\! dy-2xy dx\!\circ\! dy]/(1+x^2+y^2),
\end{array}
\end{equation}
where $\lambda$ is a non-null constant. The possible metric signatures are $(4,0)$, $(0,4)$ and $(2,2)$.

Next we show how the curvature tensor (by means of the Weyl conformal curvature operator) gives rise to a structure of symplectic pair on $(\man,g)$ with parallel opposite symplectic structure.
We construct the structures in the  positive definite case, being  the construction completely analogous in the neutral signature setting. Assume the metric signature in \eqref{GENERtype A metric} is $(++++)$
(assuming, without lost of generality, $\lambda >0$), and  consider the local orthonormal basis
\[
\begin{array}{ll}
\displaystyle
e_1=\frac{1}{\sqrt{\lambda (x^2+y^2)}}(-y\partial_x+x\partial_y),&
\quad \displaystyle
e_2=\sqrt{\frac{x^2+y^2+1}{\lambda\left(x^2+y^2\right)}}(x\partial_x+y\partial_y),\\
\noalign{\medskip}
\displaystyle
e_3=\frac{x+\sqrt{x^2+y^2}}{\sqrt{2} y \mu_{-1}
}\partial_u+\frac{1}{\sqrt{2} \mu_{-1} }\partial_v,&
\quad\displaystyle
e_4=\frac{x-\sqrt{x^2+y^2}}{\sqrt{2} y \mu_1
}\partial_u+\frac{1}{\sqrt{2} \mu_1 }\partial_v,
\end{array}
\]
where
$\mu_\varepsilon(x,y)=\sqrt{-\frac{\sqrt{x^2+y^2}}{\varepsilon
y^2+\left(x-\sqrt{x^2+y^2+1}\right)
   \left(x \varepsilon +\sqrt{x^2+y^2}\right)}}
$, $(\varepsilon=\pm 1)$.

Consider the orientation on $\man$ given by the basis $\{
e_1,\dots,e_4\}$. Then a long but straightforward calculation shows that the self-dual and anti-self-dual Weyl curvature operators can be written with respect to the orthonormal basis $\{ E_i^\pm\}$ of $\Lambda^\pm$
 given by   \eqref{PRELIMbasedosformas}  as follows:
\[
\begin{array}{ll}
W^+=\left(
\begin{array}{ccc}
 \frac{1}{2 \lambda } & 0 & 0 \\
 0 & -\frac{1}{4 \lambda } & 0 \\
 0 & 0 & -\frac{1}{4 \lambda }
\end{array}
\right), &\quad W^-=\left(
\begin{array}{ccc}
 -\frac{1}{2 \lambda } & 0 & 0 \\
 0 & \frac{1}{4 \lambda } & 0 \\
 0 & 0 & \frac{1}{4 \lambda }
\end{array}
\right).
\end{array}
\]
These  operators,  $W^\pm$, have  a distinguished eigenvalue,
$\pm\frac{1}{2 \lambda }$, whose associated eigenspace is generated by $E_1^\pm$.

Note that the local construction above depends exclusively on curvature properties, and hence  it is invariant by local isometries. As   generalized symmetric spaces are homogeneous and we assume that they are simply connected (in other case we would work at universal cover level) we have that $\operatorname{ker}(W^\pm\mp\frac{1}{2
\lambda }\id_{\Lambda^\pm})$ define a self-dual $2$-form and an anti-self-dual $2$-form,
$\Omega_\pm=\sqrt{2}E_1^\pm$, globally on $\man$.
Moreover,
\[
\Omega_+\wedge\Omega_-=0, \qquad
\Omega_+\wedge\Omega_+=-\Omega_-\wedge\Omega_-,
\]
and hence $(\Omega_+,\Omega_-)$ defines a symplectic pair on $\man$ if both forms
 $\Omega_\pm$ are closed.

We start with the self-dual $2$-form $\Omega_+ = e^1\wedge e^2 +
e^3\wedge e^4$. The $2$-form $\Omega_+$ induce an almost Hermitian structure $J_+$ on $\man$,
defined by $g(J_+X,Y)=\Omega_+(X,Y)$ and hence
$\Omega_+$ can be written in local coordinates as
\[
\Omega_+=\frac{\lambda}{\sqrt{x^2+y^2+1}}\,\, dy\wedge dx +du\wedge dv.
\]
A long but straightforward calculation shows that $\Omega_+$ is a symplectic form (i.e., $d\Omega_+=0$).
In an analogous way, the anti-self-dual $2$-form $\Omega_- =
e^1\wedge e^2 - e^3\wedge e^4$ induces an almost Hermitian structure $J_-$ on $\man$  defined by $g(J_-X,Y)=\Omega_-(X,Y)$ and hence
$\Omega_-$ can be written in local coordinates as
\[
\Omega_-=\frac{\lambda}{\sqrt{x^2+y^2+1}}dy\wedge dx -du\wedge dv.
\]
A direct computation shows that $\Omega_-$ is a symplectic structure. Furthermore, the $2$-form $\Omega_-$ is parallel and $(g,J_-)$ is an opposite K\"{a}hler structure.

In the neutral signature case one proceeds in a completely analogous way. It is, however, worth emphasizing that the orthonormal basis $\{ E^\pm_1,E^\pm_2,E^\pm_3\}$ is now of Lorentzian signature and that both one-dimensional subspaces $\operatorname{ker}(W^\pm\mp\frac{1}{2\lambda }\id_{\Lambda^\pm})$ are spacelike. Hence, the corresponding $2$-forms induce almost Hermitian structures on $(\man,g)$. Finally, observe that in the case of  neutral signature $(--++)$ the almost complex structure and the K\"{a}hler form of an almost Hermitian structure induce opposite orientations on $\man$.

It is well-known that the Ricci operator $Ric$ of any K\"{a}hler manifold is invariant by the complex structure, and hence $Ric\,J_-=J_-\, Ric$. Moreover, an explicit calculation shows that the Ricci operator of \eqref{GENERtype A metric} is given by
\[
Ric=
-\frac{3}{2 \lambda}
\left(
\begin{array}{cccc}
    1 &   0 & 0 & 0
    \\
    0 & 1 & 0 & 0
    \\
    0 & 0 & 0 & 0
    \\
    0 & 0 & 0 & 0
\end{array}
\right),
\]
and thus the Ricci operator is $J_+$-invariant (i.e.,  $Ric\, J_+=J_+\, Ric$).

\smallskip

Summarizing the above, we have:

\begin{theorem}\label{GENERth:a1}
Any four-dimensional generalized symmetric space of Type I is naturally equipped with a symplectic pair $(\Omega_+,\Omega_-)$ such that
$(\man,g,\Omega_+,\Omega_-)$ is an almost K\"{a}hler and opposite K\"{a}hler manifold with  $J_\pm$-invariant Ricci operator.
\end{theorem}

As an application of the results in \cite{BK, BCGHM} one has:

\begin{corollary}\label{GENERth:a2}
Any four-dimensional generalized symmetric space of Type I is a strictly almost K\"{a}hler manifold foliated by two complementary foliations with minimal surfaces whose induced metric is definite.
\end{corollary}

\begin{remark}\rm
It follows from the expressions of $W^\pm$ above that the self-dual and the anti-self-dual Weyl curvature operators have a distinguished eigenvalue whose corresponding one-dimensional eigenspace is spacelike. Hence (in opposition to the results in \cite{CZ2}) the underlying metric of any Type I generalized symmetric space of neutral signature  cannot admit any Walker structure with respect to any orientation, since in such a case the corresponding distinguished eigenspace should be timelike \cite{DRGRVL}.
\end{remark}

\subsection{Almost para-K\"{a}hler and opposite para-K\"{a}hler structure of generalized symmetric spaces of Type II}\label{sub:sub}

Generalized symmetric spaces of Type II have infinite order, and the corresponding pseudo-Riemannian metric is of neutral signature. The model space
$(\man,g)$ corresponds to  the space $\mathbb{R}^4$ with coordinates $(x,y,u,v)$
and metric
\begin{equation}\label{GENERtype D metric}
\begin{array}{rcl}
g&=&(\sinh(2u) - \cosh(2u)\sin(2v))dx\circ dx
\\
\noalign{\medskip}
&&
+ (\sinh(2u)+\cosh(2u) \sin(2v))dy\circ dy\\
\noalign{\medskip}
&&
-2\cosh(2u)\cos(2v)dx\circ dy+\lambda (du\circ du-\cosh^2(2u) dv\circ dv),
\end{array}
\end{equation}
where $\lambda$ is a non-zero real constant.

Next we proceed as in the previous subsection and construct a structure of symplectic pair on $(\man,g)$ by using the self-dual and anti-self-dual Weyl curvature operators. Consider the following local orthonormal frame given by ({we assume $\lambda>0$}):
\[
\begin{array}{ll}
\displaystyle
e_1=\frac{1}{\sqrt{\lambda}\cosh(2u)}\partial_v,&
\quad\displaystyle
e_2=\frac{1}{\sqrt{\frac{2e^{-2 u}}{1-\sin (2
v)}}}\left((\operatorname{sec}(2v)+\tan(2v))\partial_x+\partial_y\right),
\\
\noalign{\medskip}
\displaystyle
e_3=\frac{1}{\sqrt{\lambda}}\partial_u,&
\quad\displaystyle
e_4=\frac{1}{\sqrt{\frac{2e^{2 u}}{1+\sin (2v)}}}\left((\tan(2v)-\sec(2v))\partial_x)+\partial_y\right),
\end{array}
\]
where $e_1$, $e_2$ are timelike and $e_3$ and $e_4$ are spacelike.
Then the self-dual and anti-self-dual Weyl curvature operators can be written with respect to the orthonormal basis $\{ E_i^\pm\}$ of $\Lambda^\pm$
 defined at \eqref{PRELIMbasedosformas}   as follows:
\[
\begin{array}{ll}
W^+=\left(
\begin{array}{ccc}
 -\frac{1}{\lambda } & 0 & 0 \\
 0 & \frac{2}{ \lambda } & 0 \\
 0 & 0 & -\frac{1}{ \lambda }
\end{array}
\right), &\quad W^-=\left(
\begin{array}{ccc}
 \frac{1}{\lambda } & 0 & 0 \\
 0 & -\frac{2}{ \lambda } & 0 \\
 0 & 0 & \frac{1}{ \lambda }
\end{array}
\right).
\end{array}
\]
Both operators have a distinguished eigenvalue,
$\pm\frac{2}{\lambda}$, of multiplicity one.
Since the metric \eqref{GENERtype D metric} is of neutral signature, the induced metrics on  $\Lambda^\pm$ have Lorentzian signature. In opposition to the previously discussed case of Type I, the restriction of the metric to $\operatorname{ker}(W^\pm\mp\frac{2}{\lambda}\id_{\Lambda^\pm})$ is negative definite and hence, any local section of such subbundles is timelike. Since $(\man,g)$ is simply connected (otherwise we move to the universal cover), denote by $\Omega_\pm=\sqrt{2}E_2^\pm$ the self-dual and the anti-self-dual $2$-forms defined as sections of $\operatorname{ker}(W^\pm\mp\frac{2}{\lambda}\id_{\Lambda^\pm})$. Then $\|\Omega_\pm\|=-2$ from where it follows that  $\Omega_\pm$ are the K\"{a}hler forms of two almost para-Hermitian structures $(g,\J_\pm)$
defined by $g(\mathfrak{J}_\pm
X,Y)=\Omega_\pm(X,Y)$, which induce opposite orientations.

Now, a direct calculation shows that $\Omega_+ = e^1\wedge e^3 + e^2\wedge e^4$ expresses in local coordinates as
\[
\Omega_+=dx\wedge dy +\lambda \cosh(2u)dv\wedge du
\]
and it is closed, which shows that $(g,\J_+)$ is an almost para-K\"{a}hler structure on $\man$.

Analogously,  the anti-self-dual $2$-form $\Omega_-=e^1\wedge e^3 - e^2\wedge e^4$ expresses in local coordinates as
\[
\Omega_-=-dx\wedge dy +\lambda \cosh(2u)dv\wedge du,
\]
and it is closed as well. Moreover, it follows that $\nabla \Omega_-=0$ and hence $(g,\J_-)$ is an opposite para-K\"{a}hler structure.

Recall that the local construction above depends exclusively on curvature properties, and hence  it is invariant by local isometries. As the generalized symmetric spaces are homogeneous and we are assuming  that they are simply connected, the structures $(g,\J_\pm)$ are globally defined on $\man$.
Moreover,  a direct calculation from \eqref{GENERtype D metric} shows that the Ricci operator is given by
\[
Ric=
-\frac{6}{\lambda}
\left(
\begin{array}{cccc}
    0 & 0 & 0 & 0
    \\
     0 & 0 & 0 & 0
    \\
    0 & 0 & 1 & 0
    \\
    0 & 0 & 0 & 1
\end{array}
\right),
\]
showing that it is $\mathfrak{J}_+$-invariant.
Summarizing all the above we have:

\begin{theorem}\label{GENERth:d1}
Any four-dimensional generalized symmetric space of Type II is naturally equipped with a symplectic pair $(\Omega_+,\Omega_-)$ such that
$(\man,g,\Omega_+,\Omega_-)$ is an almost para-K\"{a}hler and opposite para-K\"{a}hler manifold with $\mathfrak{J}_\pm$-invariant Ricci operator.
\end{theorem}

Again, as an application of the results in \cite{BK, BCGHM} one has:

\begin{corollary}\label{GENERth:a2}
Any four-dimensional generalized symmetric space of Type II is a strictly almost para-K\"{a}hler manifold foliated by two complementary foliations with minimal surfaces whose induced metric is Lorentzian.
\end{corollary}

\begin{remark}\rm
Any four-dimensional generalized symmetric space of Type II is opposite para-K\"{a}hler and hence the underlying metric is a Walker one, since the $\pm 1$-eigenspaces of $\J_-$ are null parallel plane fields on $(\man,g)$
(see \cite{Brozos-Garcia-Gilkey-Nikevic-Vazquez09} for more information on Walker manifolds).
\end{remark}

\subsection{Conformal structure of generalized symmetric spaces of Type III}\label{section:2.3}

As well as for Type I, generalized symmetric spaces of Type III are of order three. The corresponding pseudo-Riemannian metric is of neutral signature and the model space
$(\man,g)$ corresponds to  the space $\mathbb{R}^4$ with coordinates $(x,y,u,v)$
and metric
\begin{equation}\label{GENERtype B metric}
g=\lambda(dx\circ dx+dy\circ dy+dx\circ  dy)+
e^{-y}(2dx+dy)\circ dv+e^{-x}(dx+2dy)\circ du,
\end{equation}
where $\lambda\in \mathbb{R}$.

In opposition to the previously discussed types I and II, the Weyl conformal curvature operator of a four-dimensional Type III generalized symmetric space does not seem to provide useful information. Indeed, assuming $\lambda> 0$, a straightforward calculation shows that
\[
\begin{array}{ll}
\displaystyle
e_1=\frac{1}{\sqrt{3\lambda}} \partial_x
-\frac{2}{\sqrt{3\lambda}}\partial_y,&
\quad\displaystyle
e_2=\frac{1}{\sqrt{\lambda}}
\partial_x, \\
 \noalign{\medskip}
 \displaystyle
 e_3=\frac{1}{\sqrt{\lambda}}
 \partial_y-e^x\sqrt{\lambda}\partial_u,&
 \quad\displaystyle
e_4=-\frac{2}{\sqrt{3\lambda}}\partial_x+\frac{1}{\sqrt{3\lambda}}
\partial_y-\frac{e^x\sqrt{\lambda}}{\sqrt{3}}\partial_u+\frac{2e^y\sqrt{\lambda}}{\sqrt{3}}\partial_v
\end{array}
\]
is a local orthonormal basis,  where $e_1$, $e_2$ are spacelike and  $e_3$, $e_4$ are timelike vector fields.
Then, the self-dual and the anti-self-dual Weyl curvature operators express in the associated orthonormal basis  \eqref{PRELIMbasedosformas} as
\[
\begin{array}{ll}
W^+=\frac{1}{3\lambda}\left(
\begin{array}{ccc}
 4 & 2 & -2 \sqrt{3} \\
 -2 & -1 & \sqrt{3} \\
 2 \sqrt{3} & \sqrt{3} & -3
\end{array}
\right), &\quad W^-=\left(
\begin{array}{ccc}
 0 & 0 & 0 \\
 0 & 0 & 0 \\
 0 & 0 & 0
\end{array}
\right),
\end{array}
\]
which shows that the self-dual Weyl curvature operator is two-step
nilpotent.

Actually, metrics \eqref{GENERtype B metric} have a distinguished conformal structure depending on the parameter $\lambda$. Recall that a pseudo-Riemannian manifold is said to be \emph{conformally symmetric} if the Weyl tensor is parallel. Moreover, while any Riemannian conformally symmetric manifold is either symmetric (i.e.,  $\nabla R=0$) or locally conformally flat (i.e.,  $W=0$), there exist non-trivial examples in the general pseudo-Riemannian setting.

\begin{theorem}\label{GENERth:d1}
Any four-dimensional generalized symmetric space of Type III is conformally symmetric. Moreover, it is locally conformally flat if and only if the parameter $\lambda$ in \eqref{GENERtype B metric} vanishes.
\end{theorem}

\begin{remark}\label{re:type C}\rm
The original classification of four-dimensional generalized symmetric spaces contained a fourth family of metrics of Lorentzian signature, whose model space corresponds to $\mathbb{R}^4$ with coordinates
$(x,y,z,t)$ and metric
\begin{equation}\label{GENERtype C metric}
g=\pm\left( e^{2t}dx\circ dx + e^{-2t}dy\circ dy + dz\circ dt  \right).
\end{equation}
These metrics correspond to  double warped product manifolds  $N\times_{f_1}\mathbb{R}\times_{f_2}\mathbb{R}$ with flat basis $N\equiv (\mathbb{R}^2(z,t), dz\circ dt)$ and warping functions $f_1(z,t)=e^{t}$, $f_2(z,t)=e^{-t}$. It was shown in \cite{Brozos-Garcia-Vazquez06} that a multiply warped product as   above is locally conformally flat if and only if
\begin{enumerate}
\item[(i)]
The conformal metric $\overline{g}_i=f_i^{-2}dz\circ dt$ is of constant sectional curvature, $i=1,2$.
\item[(ii)]
The warping functions satisfy  the compatibility condition
\[
\frac{\Delta f_1}{f_1}+\frac{\Delta f_2}{f_2}=2\frac{\langle\nabla f_1,\nabla f_2\rangle}{f_1 f_2}.
\]
\end{enumerate}
Condition (i) is trivially satisfied. On the other hand, for an arbitrary function $f(z,t)$ on $N$ we have $\Delta f=
4 \frac{\partial^2 f}{\partial z\partial t}$ and $\nabla
f=2\frac{\partial f}{\partial t} \partial_z+
2 \frac{\partial f}{\partial z} \partial_t$,
from where (ii) follows and thus the metrics given by (\ref{GENERtype C metric}) are locally conformally flat. Moreover, a straightforward calculation shows that the Ricci tensor of any metric \eqref{GENERtype C metric} is parallel and hence metrics \eqref{GENERtype C metric} are locally symmetric. Furthermore, they are Cahen-Wallach symmetric spaces so we skip them in what follows  since we are interested in the non-symmetric cases.
\end{remark}

\section{Characterizations}\label{section-characterizations}

It is well-known that any four-dimensional (complete and simply connected) homogeneous Riemannian manifold is either symmetric or a Lie group. Although a similar result is not true in full generality for Lorentzian or neutral signature (see, for example \cite{CZ2}), all four-dimensional generalized symmetric spaces share this property, a fact which follows as a consequence of the results below.

\subsection{Characterizing Type I generalized symmetric spaces}

Riemannian four-dimensional manifolds admitting a strictly almost K\"{a}hler structure $(g,J_+)$ and opposite K\"{a}hler structure $(g,J_-)$ such that the Ricci operator is $J_+$-invariant have been investigated in \cite{Apostolov-Armstrong-Draghici02} showing that they are homogeneous and locally isometric to the unique Riemannian generalized symmetric space (see also \cite{Fino05}). This result has been extended by the authors in \cite{Calvino-Garcia-Vazquez-Vazquez3} to the neutral signature case, thus resulting in the following characterization of Type I generalized symmetric spaces. (Recall that the roles of $J_+$ and $J_-$ must   be changed in the following theorem when considering the neutral signature case, since $J_\pm$ and $\Omega_\pm$ induce opposite orientations in that setting).

\begin{theorem}\label{GENteorema de Fino}{\rm \cite{Apostolov-Armstrong-Draghici02,Calvino-Garcia-Vazquez-Vazquez3}}
Let $(\man,g)$ be a four-dimensional pseudo-Riemannian manifold. Then $(\man,g)$ is locally isometric to a Type I generalized symmetric space if and only if $(\man,g)$ admits a strictly almost K\"{a}hler structure $(g,J_+)$ and opposite K\"{a}hler structure $(g,J_-)$ such that the Ricci operator is $J_+$-invariant.
\end{theorem}

\begin{remark}\label{re:remarkI}\rm
Any Type I generalized symmetric space is locally isometric to a Lie group $G$  whose Lie algebra is
given  by
$\mathfrak{g}=\operatorname{Span}\{ e_1,e_2,e_3,e_4\}$ with non-zero brackets
\[
\begin{array}{lll}
{[e_1,e_3]}=e_1+2\alpha e_2,&
{[e_1,e_4]}=\alpha e_1,&
{[e_2,e_3]}=-e_2,\\
\noalign{\medskip}
{[e_2,e_4]}=-2 e_1-\alpha e_2,&
{[e_3,e_4]}=-2\alpha e_3-2 e_4,
\end{array}
\]
and the left-invariant metric which makes $\{ e_1,e_2,e_3,e_4\}$ an orthonormal basis of signature either $(++++)$ or $(--++)$.

Moreover, Type I generalized symmetric spaces are characterized by the existence of an almost K\"{a}hler structure $(g,J_+)$ and an opposite K\"{a}hler structure $(g,J_-)$ with $J_+$-invariant Ricci operator so that the curvature tensor is $J_+$-invariant (i.e., $R(J_+X,J_+Y,J_+Z,J_+V)=R(X,Y,Z,V)$).
\end{remark}

\begin{remark}
\rm
As a consequence of the previous analysis, all the almost K\"{a}hler and opposite K\"{a}hler structures constructed in \cite{Calvaruso} are locally isomorphic.
\end{remark}

\subsection{Type II generalized symmetric spaces}

First of all, it is worth emphasizing that a characterization result like that of Theorem \ref{GENteorema de Fino} cannot be expected for Type II generalized symmetric spaces. This is due to the  existence of  four-dimensional manifolds $(\man,g)$ which are not locally homogeneous and which admit a strictly almost para-K\"{a}hler structure $(g,\J_+)$ and opposite para-K\"{a}hler structure $(g,\J_-)$ with $\J_+$-invariant Ricci operator.

Moreover, even in the homogeneous case, the authors showed in \cite{Calvino-Garcia-Vazquez-Vazquez3} the existence of  strictly almost para-K\"{a}hler structures and opposite para-K\"{a}hler structures with invariant Ricci operator that are not locally isometric to any Type II generalized symmetric spaces.
A detailed examination of four-dimensional Lie groups admitting a pair of structures as   above leads to the following possibilities.

\begin{lemma}\label{thm:algebraic possibilities Type D}
Let $G$ be a four-dimensional Lie group and $g$ a left-invariant metric on $G$. Let $(\Omega_+,\Omega_-)$ be a symplectic pair induced by almost para-K\"{a}hler structures $(\mathfrak{J}_+,\mathfrak{J}_-)$ on $(G,g)$  such that $\Omega_-$ is parallel and the Ricci operator is $\mathfrak{J}_+$-invariant. Then one of the following possibilities holds:
\begin{itemize}
\item[(i)] The self-dual and anti-self-dual Weyl curvature operators are diagonalizable.
\item[(ii)] The minimal polynomial of the  self-dual Weyl curvature operator has a double root.
\item[(iii)] The metric is self-dual and the self-dual Weyl curvature operator is two-step nilpotent and non-flat.
\end{itemize}
\end{lemma}

\begin{proof}
Let $(\Omega_+,\Omega_-)$ be a symplectic pair on $(G,g)$. Then there exist two complementary minimal foliations $\mathcal{F}$ and $\mathcal{G}$ defined by   $2$-forms $\omega_\mathcal{F}$ and
$\omega_\mathcal{G}$   satisfying
$\Omega_\pm=\frac{1}{2}(\omega_\mathcal{F} \pm \omega_\mathcal{G})$. Let  $\{e_1,e_2,e_3,e_4\}$ be an orthonormal basis of the Lie algebra $\mathfrak{g}$ with
$e_1$, $e_2$ timelike and  $e_3$, $e_4$ spacelike, such that $\Omega_\pm = e^1\wedge e^3 \pm e^2\wedge e^4$ or, equivalently, the structures $\mathfrak{J}_-$ and $\mathfrak{J}_+$ express in the above basis as
\[\mathfrak{J}_-=\left(
\begin{array}{cccc}
 0 & 0 & 1 & 0 \\
 0 & 0 & 0 & -1 \\
 1 & 0 & 0 & 0 \\
 0 & -1 & 0 & 0
\end{array}
\right),\qquad \mathfrak{J}_+=\left(
\begin{array}{cccc}
 0 & 0 & 1 & 0 \\
 0 & 0 & 0 & 1 \\
 1 & 0 & 0 & 0 \\
 0 & 1 & 0 & 0
\end{array}
\right).
\]
Moreover, since the corresponding distributions $\mathcal{G}=\operatorname{span}\{ e_1,e_3\}$ and $\mathcal{F}$ $=$ $\operatorname{span}\{ e_2,e_4\}$ are integrable, the Lie brackets can be expressed as follows
\begin{equation}
\label{equation-1}
\left\{
\begin{array}{l}
[e_1,e_2]=a_1  e_1 + b_1  e_2 + c_1  e_3 + d_1  e_4,\\
\noalign{\medskip}
[e_1,e_3]=a_2 e_1+ c_2 e_3,\\
\noalign{\medskip}
[e_1,e_4]=a_3 e_1+ b_3 e_2+c_3 e_3+d_3 e_4,\\
\noalign{\medskip}
[e_2,e_3]=a_4 e_1+ b_4 e_2+c_4 e_3+d_4 e_4,\\
\noalign{\medskip}
[e_2,e_4]= b_5 e_2+d_5 e_4,\\
\noalign{\medskip}
[e_3,e_4]=a_6  e_1 + b_6  e_2 + c_6  e_3 + d_6  e_4.
\end{array}\right.
\end{equation}

Now, the fact that both foliations $\mathcal{F}$ and $\mathcal{G}$  have minimal leaves implies that $c_4=a_1$, $c_6=-a_3$, $d_3=-b_1$ and $d_6=b_4$.

Since we are assuming that the opposite structure $\mathfrak{J}_-$ is para-K\"{a}hler, it must be integrable (equivalently, its Nijenhuis tensor vanishes identically). This condition translates  to the Lie algebra as $c_1=-2a_3-a_4$, $c_3=-2a_1+a_6$, $d_1=b_3+2b_4$, and $d_4=2b_1-b_6$.

Summarizing the above, the expression of the brackets in \eqref{equation-1} reduces to

\begin{equation}
\label{equation-2}
\left\{
\begin{array}{l}
[e_1,e_2]=a_1  e_1 + b_1  e_2 -(2a_3+a_4)  e_3 + (b_3+2b_4)  e_4,\\
\noalign{\medskip}
[e_1,e_3]=a_2 e_1+ c_2 e_3,\\
\noalign{\medskip}
[e_1,e_4]=a_3 e_1+ b_3 e_2+(a_6-2a_1) e_3-b_1 e_4,\\
\noalign{\medskip}
[e_2,e_3]=a_4 e_1+ b_4 e_2+a_1 e_3+(2b_1-b_6) e_4,\\
\noalign{\medskip}
[e_2,e_4]= b_5 e_2+d_5 e_4,\\
\noalign{\medskip}
[e_3,e_4]=a_6  e_1 + b_6  e_2 -a_3  e_3 +b_4  e_4.
\end{array}\right.
\end{equation}
A standard calculation  shows that the almost para-K\"ahler structure $\J_+$ is integrable if and only if the constants $a_1$, $a_3$, $b_1$ and $b_4$ vanish. Hence, we assume that $a_1$, $a_3$, $b_1$, and  $b_4$ do not vanish  simultaneously.

Now, observe that \eqref{equation-2} defines a Lie algebra if and only if (just verifying the Jacobi identity):

\begin{equation}\label{equation-3}
\begin{array}{rll}
a_1 a_2 + a_1 b_4 + a_6 (b_3 + 2 b_4) + a_3 (2 b_1 - b_6) - a_4 (b_1 + c_2)&=&0,\\
\noalign{\medskip}
a_2 b_1 + 2 b_1 b_3 + 2 b_4 b_6 - b_4 c_2&=&0,\\
\noalign{\medskip}
a_1 (2 a_4 + b_5) + a_3 (2 a_6 + d_5)&=&0,\\
\noalign{\medskip}
a_4 b_5-2 a_3 a_4 - 2 a_1 a_6  - a_6 d_5&=&0,\\
\noalign{\medskip}
b_4 b_5-a_6 b_1 - a_4 b_3 - a_3 b_4  - a_1 b_6 - b_6 d_5&=&0,\\
\noalign{\medskip}
b_4 b_5 -a_2 a_3 - a_6 c_2 - b_6 d_5&=&0,\\
\noalign{\medskip}
2 b_3 b_4-a_2 b_3  + 2 b_1 b_6 - b_6 c_2&=&0,\\
\noalign{\medskip}
 b_4 (a_2 + 2 (b_3 + b_4))-2 b_1^2  + b_1 (2 b_6 - c_2)&=&0,\\
\noalign{\medskip}
2 a_3^2-2 a_1^2  + a_3 b_5 + a_4 b_5 + a_1 d_5 - a_6 d_5&=&0,\\
\noalign{\medskip}
2 a_1 b_4-a_1 b_3  - a_6 b_4 + b_1 b_5 + a_4 b_6 + a_3 (b_1 + 2 b_6) + b_3 d_5&=&0,\\
\noalign{\medskip}
a_2 a_6-2 a_1 a_2  - b_1 b_5 - a_3 c_2 - b_3 d_5&=&0,\\
\noalign{\medskip}
a_6 (b_3 + 2 b_4) - b_1 (a_4-2 a_3 + 2 b_5) + b_6(b_5-a_3)  + b_4 (a_1 + d_5)&=&0,\\
\noalign{\medskip}
b_1 d_5-5 a_1 b_1 - a_4 (2 b_3 + b_4) - a_3 (b_3 + 6 b_4) - b_3 b_5  -
 b_6 (a_6 + 2 d_5)&=&0,\\
\noalign{\medskip}
a_2 (2 a_3 + a_4) - 2 a_6 b_1 + a_3 b_3 + 4 a_3 b_4 + a_4 b_4 + a_6 b_6 + a_1 (5 b_1 - 2 b_6 + c_2)&=&0.
\end{array}
\end{equation}

The proof now follows by a computed-aided calculation. All possible solutions of \eqref{equation-3} can be explicitly obtained showing the structure of the self-dual and anti-self-dual Weyl  curvature operators as in the statement.
\end{proof}

\begin{remark}\label{re:remark}
\rm
Illustrating the previous result, recall from \cite{Calvino-Garcia-Vazquez-Vazquez3} the following simple examples corresponding to the different possibilities in Lemma \ref{thm:algebraic possibilities Type D}:
\begin{enumerate}
\item
Consider the four-dimensional Lie algebra given by the non-zero brackets
\[
\qquad [e_1,e_2]=- 2\alpha e_3,\quad
[e_1,e_4]=\alpha e_1, \quad
[e_2,e_4]=-2 \alpha e_2,\quad
[e_3,e_4]=-\alpha e_3.
\]
Then $(\Omega_+,\Omega_-)$ defined as in the proof of Lemma \ref{thm:algebraic possibilities Type D} is a strictly almost para-K\"ahler and opposite
para-K\"ahler structure if and only if  $\alpha\neq0$ with commuting Ricci operator. The self-dual and anti-self-dual Weyl curvature operators are
given by
\[ W^+=\left(
\begin{array}{ccc}
 -\alpha^2 & 0 & 0 \\
 0 & 2\alpha^2  & 0 \\
 0 & 0 & -\alpha^2
\end{array}
\right), \,\,\,
W^-=\left(
\begin{array}{ccc}
 \alpha^2 & 0 & 0 \\
 0 & -2\alpha^2  & 0 \\
 0 & 0 & \alpha^2
\end{array}
\right).
\]

\item
Consider the four-dimensional Lie algebra given by the non-zero brackets
\[\qquad
\begin{array}{ll}
[e_1,e_2]=- e_1 - e_3,&
[e_1,e_4]= e_1 + (\alpha+2) e_3,\quad
[e_2,e_3]=- e_1 - e_3,\\
\noalign{\medskip}
[e_2,e_4]=2 (\alpha +1) e_2,&
[e_3,e_4]=\alpha e_1 - e_3.
\end{array}
\]
Then $(\Omega_+,\Omega_-)$ defined as in the proof of Lemma \ref{thm:algebraic possibilities Type D} is a strictly almost para-K\"ahler and opposite
para-K\"ahler structure with commuting Ricci operator. Now the self-dual and anti-self-dual Weyl curvature operators are
\[
W^+= \left(
\begin{array}{ccc}
 \frac{2}{3} \left(\alpha^2-4 \alpha-5\right) & 0 & 4 (\alpha+1)\\
 0 & -\frac{4}{3} (\alpha+1)^2 & 0 \\
 -4 (\alpha+1) & 0 & \frac{2}{3} \left(\alpha^2+8  \alpha+7\right)
\end{array}
\right),
\]
and
\[ W^-= \left(
\begin{array}{ccc}
 \frac{2}{3} (\alpha+1)^2 & 0 & 0 \\
 0 & -\frac{4}{3} (\alpha+1)^2  & 0 \\
 0 & 0 & \frac{2}{3} (\alpha+1)^2
\end{array}
\right).
\]
Observe   that $W^+$ is diagonalizable if and only if $\alpha=-1$, in which case the manifold is flat.

\item
Consider the four-dimensional Lie algebra given by the non-zero brackets
\[
\begin{array}{ll}
[e_1,e_2]=- e_1 - e_3,&
[e_1,e_4]= e_1 + e_3,\quad
[e_2,e_3]=- e_1 - e_3,\\
\noalign{\medskip}
[e_2,e_4]=\alpha e_2 + \alpha e_4,&
[e_3,e_4]=-e_1-e_3.
\end{array}
\]
Then $(\Omega_+,\Omega_-)$ defined as in the proof of Lemma \ref{thm:algebraic possibilities Type D} is a strictly almost para-K\"ahler and opposite
para-K\"ahler structure with commuting Ricci operator. Moreover, the induced metric is self-dual (i.e.,  $W^-=0$) and Ricci flat, with
self-dual Weyl curvature operator given by
\[
W^+=\left(
\begin{array}{ccc}
 -2 \alpha & 0 & 2 \alpha \\
 0 & 0 & 0 \\
 -2 \alpha & 0 & 2 \alpha
\end{array}
\right),
\]
which shows that $W^+$ is two-step nilpotent (and non-flat for any $\alpha\neq 0$).

\end{enumerate}
\end{remark}

 Lemma \ref{thm:algebraic possibilities Type D} shows that one may not expect to characterize   Type II generalized symmetric spaces by the existence of an almost para-K\"{a}hler and opposite para-K\"{a}hler structure with invariant Ricci operator, even when restricting to the homogeneous setting. Indeed an additional condition on the diagonalizability of the self-dual and anti-self-dual Weyl curvature operators must be assumed.

\begin{remark}\rm
Considering those almost Hermitian manifolds whose curvature tensor resembles that of a K\"{a}hler manifold, Gray introduced the following identities:
$$
\begin{array}{rl}
(G_1)& R(X,Y,Z,V)=R(X,Y,JZ,JV),\\
(G_2)& R(X,Y,Z,V)-R(JX,JY,Z,V)=R(JX,Y,JZ,V)+R(JX,Y,Z,JV),\\
(G_3)& R(X,Y,Z,V)=R(JX,JY,JZ,JV).
\end{array}
$$
Obviously, $(G_3)\Rightarrow (G_2)\Rightarrow (G_1)$. Moreover, a four-dimensional almost K\"{a}hler manifold satisfies $(G_3)$ if and only if the K\"{a}hler form $\Omega$ is an eigenvector of the curvature operator acting on $\Lambda^2$ and the Ricci tensor is $J$-invariant. Further,  condition $(G_2)$ means that the Ricci operator and the self-dual Weyl curvature operator have the same symmetries as they have in the K\"{a}hler case. In fact, a four-dimensional almost K\"{a}hler structure $(g,J)$ satisfying $(G_3)$ also satisfies $(G_2)$ if and only if the norm of the covariant derivative of the K\"{a}hler form, $\|\nabla\Omega\|$, is constant \cite{Apostolov-Armstrong-Draghici02}.

The identities above take the following form in the para-Hermitian setting:
$$
\begin{array}{rl}
(\G_1)& R(X,Y,Z,V)=-R(X,Y,\J Z,\J V),\\
(\G_2)& R(X,Y,Z,V)+R(\J X,\J Y,Z,V)=-R(\J X,Y,\J Z,V)-R(\J X,Y,Z,\J V),\\
(\G_3)& R(X,Y,Z,V)=R(\J X,\J Y,\J Z,\J V).
\end{array}
$$
Now, observe that all Lie groups in Remark \ref{re:remark} satisfy the identity $(\G_3)$, but only the Lie group in Remark \ref{re:remark}-$(1)$ satisfy $(\G_2)$.
This shows that the constancy of $\|\nabla\Omega\|$ is not a sufficient condition for an almost para-K\"{a}hler $(\G_3)$-manifold to be $(\G_2)$ in contrast with the corresponding almost Hermitian situation (see \cite{Apostolov-Armstrong-Draghici02}).

Closely related, observe that Lie groups corresponding to Remark \ref{re:remark}-$(1)$ have
$\|\nabla\Omega\|^2=\frac{8}{3}\tau\neq 0$, while the corresponding almost para-K\"{a}hler structures corresponding to cases $(2)-(3)$ in Remark \ref{re:remark} are isotropic K\"{a}hler (i.e., $\|\nabla\Omega\|=0$)
\end{remark}

In order to obtain  our results we will extensively use the following criterion for local conformal equivalence of pseudo-Riemannian metrics, which immediately follows from the work of  Hall  \cite{Hall10}.
It is well-known that the Weyl conformal curvature tensor of a pseudo-Riemannian manifold only depends on the conformal class of the metric. Indeed, if $g$ and $\overline{g}=e^{2\sigma}g$ are conformally related metrics, then    the associated  $(1,3)$-Weyl conformal curvature tensors $W$ and $\overline{W}$ coincide with each other (equivalently, the $(0,4)$-Weyl conformal curvature tensors satisfy $\overline{W}=e^{2\sigma}W$). A possible converse of the above statement holds true under some non-degeneracy conditions on the Weyl tensor as follows.

\begin{lemma}{\rm\cite{Hall10}}\label{lemma-Hall}
Let $(M,g)$ be a pseudo-Riemannian manifold whose Weyl conformal curvature operator has maximal rank. Then any other metric $\overline{g}$ on $M$ with the same $(1,3)$-Weyl conformal curvature tensor  is locally conformally equivalent to the metric $g$.
\end{lemma}

Now, as an application of Lemma \ref{thm:algebraic possibilities Type D} and Lemma \ref{lemma-Hall} we have the following:

\begin{theorem}\label{thm:characterization Type D}
Let $G$ be a four-dimensional Lie group and $g$ a left-invariant metric on $G$. Then $(G,g)$ is locally isometric (up to homothety) to a Type II generalized symmetric space if and only if $(G,g)$ admits a strictly almost para-K\"{a}hler structure $\J_+$ and an opposite para-K\"{a}hler structure $\J_-$ such that the Ricci tensor is
$\J_+$-invariant and the self-dual and the anti-self-dual Weyl curvature operators $W^\pm$ are non-zero and diagonalizable.
\end{theorem}

\begin{proof} 
First of all recall that the scalar curvature of any Type II generalized symmetric space is given by $\tau=-\frac{12}{\lambda}$. A straightforward calculation shows that the self-dual and anti-self-dual Weyl curvature operators of any Lie group corresponding to Lemma \ref{thm:algebraic possibilities Type D}--(i) are given by
$$
W^+ = \frac{\tau}{12}\operatorname{diag}[1,-2,1]
\qquad
\mbox{and}\quad
W^- = \frac{\tau}{12}\operatorname{diag}[-1,2,-1].
$$
Moreover, the Ricci operator of any such Lie group is diagonalizable of the form $Ric=\frac{\tau}{2}\operatorname{diag}[0,0,1,1]$. This shows that the Weyl curvature operator is of maximal rank unless the manifold is flat, and thus Lemma \ref{lemma-Hall} implies  that all Lie groups given by Lemma \ref{thm:algebraic possibilities Type D}--(1) are conformally equivalent to the generalized symmetric space \eqref{GENERtype D metric}.

Next fix the scalar curvature of a Lie group as   above to be the scalar curvature of a metric \eqref{GENERtype D metric}, which is given by $\tau=-\frac{12}{\lambda}$.

Let $g$ and $\overline{g}=e^{2\sigma}g$ be two conformally related metrics and let $\{ e_i\}$ and $\{\overline{e}_i=e^{-\sigma}e_i\}$ be local orthonormal basis. Then the corresponding basis $\{ E^\pm_i\}$ and $\{\overline{E}^\pm_i\}$ of $\Lambda^\pm$ are related by $\overline{E}^\pm_i=e^{2\sigma}E^\pm_i$. Since the self-dual and anti-self-dual Weyl curvature operators of $g$ and $\overline{g}$ coincide, both have $\pm \frac{\tau}{6}$ as a distinguished eigenvalue. Considering local sections $\Omega_\pm$ and  $\overline{\Omega}_\pm=e^{2\sigma}\Omega_\pm$, define almost para-complex structures  $\J_\pm$ and $\overline{\J}_\pm$  as in
Section \ref{sub:sub}, which are related as follows:
$$
\begin{array}{rcl}
\overline{g}(\overline{\J}_\pm X,Y)&=&\overline{\Omega}_\pm (X,Y)=e^{2\sigma}\Omega_\pm (X,Y)=e^{2\sigma}g(\J_\pm X,Y)\\
\noalign{\medskip}
&=&\overline{g}(\J_\pm X,Y),
\end{array}
$$
which shows that the Weyl curvature tensors of any two conformally related metrics $g$ and $\overline{g}=e^{2\sigma}g$ determine the same almost para-complex structures $\J_\pm$.
Now, since $(g,\J_-)$ and $(\overline{g},\J_-)$ are para-K\"ahler, it immediately follows that the conformal factor is an homothety and hence an isometry for any fixed value of the scalar curvature.
\end{proof}

\begin{remark}\rm\label{remark111}
Alternatively, observe that the non-zero entries of the $(0,4)$-Weyl curvature tensor of any Lie group as in the above Theorem are given (up to the usual symmetries) by
$W_{1234}=W_{1423}=-\frac{1}{\lambda}$ and
$W_{1324}=-\frac{2}{\lambda}$.
Now, a straightforward calculation shows that the norm of the Weyl tensor of any such  Lie group  agrees with the norm of the Weyl tensor of the Type II generalized symmetric space defined by \eqref{GENERtype D metric}, which becomes
$
\| W\|^2 =g^{i\alpha}g^{j\beta}g^{k\gamma}g^{\ell\delta}W_{ijk \ell}W_{\alpha\beta\gamma\delta}=\frac{48}{\lambda^2},
$
from where it follows that the conformal factor is necessarily constant, and thus, all such Lie groups $(G,g)$ are homothetic.
\end{remark}

Proceeding in a completely analogous way, all Lie groups $(G,g)$ as in Lemma \ref{thm:algebraic possibilities Type D}--(ii) are also homothetically equivalent.

\begin{theorem}\label{thm:characterization Type D2}
Any four-dimensional Lie group $G$ with a left-invariant metric  $g$ admitting a strictly almost para-K\"{a}hler structure $(g,\J_+)$ and an opposite para-K\"{a}hler structure $(g,\J_-)$ such that the Ricci tensor is
$\J_+$-invariant and such that the minimal polynomial of
the self-dual Weyl curvature operator has a non-zero double root are locally homothetic.
\end{theorem}

\begin{proof}
Being an opposite para-K\"ahler manifold, the anti-self-dual Weyl curvature operator is diagonalizable of the form $W^- = \frac{\tau}{12}\operatorname{diag}[-1,2,-1]$. Moreover, a straightforward calculation shows that the self-dual Weyl curvature operator of any Lie group corresponding to Lemma \ref{thm:algebraic possibilities Type D}--(ii) has the same eigenvalue structure that $W^-$, i.e.  $W^+$ has eigenvalues $\frac{\tau}{6}$ and $-\frac{\tau}{12}$, the latter with multiplicity two being a double root of the minimal polynomial.
Hence, the Weyl curvature operator is of maximal rank unless the scalar curvature vanishes (in which case the manifold is Ricci flat and self-dual). Then Lemma \ref{lemma-Hall} shows that any two
such Lie groups are  conformally equivalent and the proof follows as in the previous theorem.
\end{proof}

\begin{remark}\rm
Considering the Weyl curvature operator of any Lie group as in Lemma \ref{thm:algebraic possibilities Type D}--(ii), it decomposes as $W=W^++W^-$ and hence $\| W\|^2=\| W^+\|^2 + \| W^-\|^2$. Now, since the Lie group admits an opposite para-K\"ahler structure, $\| W^-\|^2=\frac{1}{24}\tau^2$. Moreover, since the self-dual Weyl curvature operator has the same eigenvalues as $W^-$ but with opposite sign and it has a double root of its minimal polynomial, a straightforward calculation shows that $\| W^+\|^2=\frac{1}{24}\tau^2$. Hence $\| W\|^2=\frac{1}{12}\tau^2\neq 0$. Hence, any two conformally related such metrics must be homothetic.
\end{remark}


The remaining class of almost para-K\"ahler and opposite para-K\"ahler structures at Lemma \ref{thm:algebraic possibilities Type D}-(iii) are self-dual and Einstein (indeed Ricci flat). Hence they are Osserman manifolds and the corresponding Jacobi operators are two-step nilpotent. Since any para-K\"{a}hler metric is a Walker one, it follows  (see for example \cite{Brozos-Garcia-Gilkey-Nikevic-Vazquez09} for the discussion of Ricci flat self-dual Walker metrics) that $(G,g)$ is locally a Riemannian extension \eqref{eq:Riemannianextension} of a flat affine surface $(\affman,\affcon)$. Moreover, in such a case the Walker structure is strict (i.e.,  the parallel degenerate distribution is spanned by two parallel null vector fields), and one has the following:

\begin{theorem}
Any four-dimensional Lie group $G$ with a left-invariant metric  $g$ admitting a strictly almost para-K\"{a}hler structure $(g,\J_+)$ and an opposite para-K\"{a}hler structure $(g,\J_-)$ such that the Ricci tensor is $\J_+$-invariant and the self-dual Weyl curvature operator vanishes is locally a Riemannian extension $(T^*\affman,g_{\affcon,\Phi})$ of a flat affine surface $(\affman,\affcon)$. Hence, there exist coordinates $(x_1,x_2,y_1,y_2)$ where the metric expresses as
$$
g=2 \,dx_i\circ dy_i + f(x_1,x_2) dy_2\circ dy_2
$$
for some function $f(x_1,x_2)$ defined on $\affman$.
\end{theorem}

\subsection{Characterizing Type III generalized symmetric spaces}
Finally we consider the third class of four-dimensional generalized symmetric spaces from the point of view of Theorem \ref{GENERth:d1}, which shows that the constant $\lambda$ plays an important role, distinguishing the locally conformally flat examples.

\smallskip

The case $\lambda\neq 0$ is settled as follows:

\begin{theorem}\label{theorema:1} Let $(\man,g)$ be a four-dimensional non-symmetric homogeneous space of neutral signature. Then $(\man,g)$ is locally isometric to the generalized symmetric space of Type III if and only if $(\man,g)$ is strictly conformally symmetric with negative  semidefinite  Ricci tensor and rank $2$ Ricci operator.
\end{theorem}

\begin{proof}
Recall from Theorem \ref{GENERth:d1} that any Type III generalized symmetric space with $\lambda\neq 0$ is conformally symmetric. It was shown in \cite{Derdzinski-Roter80} that the Ricci operator of any strictly conformally symmetric manifold (i.e.,  $\nabla W=0$ but $W\neq 0$ and $\nabla R\neq 0$) satisfies $\operatorname{rank}(Ric)\leq 2$. Moreover, the Weyl curvature tensor is completely determined by the Ricci tensor  as $FW=-\frac{1}{2}\rho\odot\rho$, where $\odot$ is the Kulkarni-Nomizu product and $F$ is a function which vanishes at those points where $\operatorname{rank}(Ric)\leq 1$.

Now, it follows from  \eqref{GENERtype B metric} that the Ricci operator, $Ric$, and the Ricci tensor, $\rho$, are given by
\[
Ric=\left(\begin{array}{cccc}
0&0&0&0\\
0&0&0&0\\
0&-\frac{4}{3}e^x&0&0\\
-\frac{4}{3}e^y&0&0&0
\end{array}\right),
\quad
\rho=\left(\begin{array}{cccc}
-\frac{4}{3}&-\frac{2}{3}&0&0\\
-\frac{2}{3}&-\frac{4}{3}&0&0\\
0&0&0&0\\
0&0&0&0
\end{array}\right),
\]
which shows that $Ric$ is two-step nilpotent and that $\rho$ is negative semidefinite.
Moreover,  a straightforward calculation  from \eqref{GENERtype B metric} shows that the Weyl tensor is given by  $W=\frac{3\lambda}{16}\rho\odot\rho$, and hence
$FW=-\frac{1}{2}\rho\odot\rho$ for the function $F$ given by $F=-\frac{8}{3\lambda}$. This shows that the Ricci operator is of rank $2$.

Finally, note that Derdzinski showed in \cite{Derdzinski78} the existence of a unique conformally symmetric homogeneous space whose Ricci operator has rank $2$ and whose corresponding Ricci tensor is  negative semidefinite, and this ends the proof.
\end{proof}

It follows from the work in \cite{Derdzinski78} that the metric of any homogeneous conformally symmetric as in Theorem \ref{theorema:1}  is locally isometric to a Riemannian extension $(T^*\Sigma, g_{D,\Phi})$, where $(\Sigma, D)$  is a projectively flat affine surface and $\Phi$ is a symmetric $(0,2)$-tensor field on $\Sigma$. We summarize this observation in the following

\begin{theorem}{\rm\cite{Derdzinski-Roter97}}\label{th.esteb}
Any conformally symmetric generalized symmetric space of Type III is locally isometric to the   Riemannian extension $(T^\ast\Sigma, g_{D,\Phi})$ where $(\Sigma,D)$ is a projectively flat surface given by the Christoffel symbols
 $$
\Gamma_{11}^1=-e^{-\frac{1}{3}\ln 4},
\quad
\Gamma_{12}^2=\Gamma_{21}^2=e^{-\frac{1}{3}\ln 4},
\quad
\Gamma_{22}^1=e^{-\frac{1}{3}\ln 4},
$$
 and $\Phi$ is a $(0,2)$-tensor on $\Sigma$ given by $\Phi=-\frac{3}{16}\lambda e^{\frac{7}{6}\ln 4}\operatorname{Id}$.

\end{theorem}

Recall that a pseudo-Riemannian manifold is said to be \emph{semi-symmetric} if the curvature tensor at each point is that of a symmetric space, possibly changing from point to point. This condition is equivalent to $R(X,Y)\cdot R=0$, where $R(X,Y)$ acts on $R$ as a derivation, and a
direct calculation shows that a  Type III generalized symmetric space is semi-symmetric if and only if it is locally conformally flat. That is the reason why the treatment of the case $\lambda=0$  is completely different to the case  $\lambda\neq 0$ studied above.

\begin{theorem}\label{th:esteb2}
Any  locally conformally flat generalized symmetric space of Type III is the Riemannian  extension $(T^\ast\Sigma,g_\affcon)$ of an affine surface $(\Sigma,\affcon)$, where $\affcon$ is a projectively flat affine connection with symmetric and non-degenerate Ricci tensor.
\end{theorem}

\begin{proof}
A direct   calculation shows that the distribution  $\D=\operatorname{Span}\{\partial_u,\partial_v\}$ in the generalized symmetric space of Type III is null and parallel, and therefore this space is a Walker manifold \cite{Brozos-Garcia-Gilkey-Nikevic-Vazquez09}. Moreover,    it follows from  \cite{Afifi54}  that this space is a Riemannian  extension of a surface $\Sigma$ equipped with a projectively flat affine connection $\affcon$. Now, a long but straightforward calculation  shows that $R(a,b) R(c,d)= R(c,d) R(a,b)$ for any vector fields $a$, $b$, $c$, $d$, and thus the Ricci tensor of the affine surface $\Sigma$ is symmetric \cite{Brozos-Garcia-Gilkey-Vazquez07}. Finally, it is easy to show that $(T^\ast\Sigma,g_\affcon)$ is not Ivanov-Petrova and hence the Ricci tensor of the affine surface $(\Sigma,D)$ is not degenerate (we refer to \cite{Calvino-Garcia-Vazquez10} for more details).
\end{proof}

\section{Appendix}

Because of some incorrectness we have encountered in recent studies investigating the geometry of four-dimensional generalized symmetric spaces, in what follows we give an explicit description of the Levi-Civita connection and the $(0,4)$-curvature tensor corresponding to the  different generalized symmetric spaces.

\subsection*{Type I}
It follows from Theorem \ref{GENteorema de Fino} and Remark \ref{re:remarkI} that any
Type I generalized symmetric space is locally isometric to a metric Lie group $(G,g)$  whose Lie algebra is given  by $\mathfrak{g}=\operatorname{Span}\{ e_1,e_2,e_3,e_4\}$ with non-zero brackets
\[
\begin{array}{lll}
{[e_1,e_3]}=e_1+2\alpha e_2,&
{[e_1,e_4]}=\alpha e_1,&
{[e_2,e_3]}=-e_2,\\
\noalign{\medskip}
{[e_2,e_4]}=-2 e_1-\alpha e_2,&
{[e_3,e_4]}=-2\alpha e_3-2 e_4,
\end{array}
\]
and the left-invariant metric which makes $\{ e_1,e_2,e_3,e_4\}$ an orthonormal basis of signature either $(++++)$ or $(--++)$. Both cases are described simultaneously by setting $\eta=+1$ (resp., $\eta=-1$) in the Riemannian (resp., pseudo-Riemannian) case.
Now, a straightforward calculation shows that the Levi-Civita connection of any Type I generalized symmetric space is given (with respect to a suitable orthonormal basis of left-invariant vector fields on $G$) by
$$
{\small
\begin{array}{llll}
\nabla_{e_1}e_1=-\eta(e_3+\alpha e_4)
&
\nabla_{e_1}e_2=\eta(e_4-\alpha e_3)
&
\nabla_{e_1}e_3=e_1+\alpha e_2
&
\nabla_{e_1}e_4=\alpha e_1-e_2
\\
\noalign{\medskip}
\nabla_{e_2}e_1=\eta(e_4-\alpha e_3)
&
\nabla_{e_2}e_2=\eta(e_3+\alpha e_4)
&
\nabla_{e_2}e_3=\alpha e_1-e_2
&
\nabla_{e_2}e_4=-e_1-\alpha e_2
\\
\noalign{\medskip}
\nabla_{e_3}e_1=-\alpha e_2
&
\nabla_{e_3}e_2=\alpha e_1
&
\nabla_{e_3}e_3=2\alpha e_4
&
\nabla_{e_3}e_4=-2\alpha e_3
\\
\noalign{\medskip}
\nabla_{e_4}e_1= - e_2
&
\nabla_{e_4}e_2= e_1
&
\nabla_{e_4}e_3= 2 e_4
&
\nabla_{e_4}e_4=-2e_3
\end{array}
 }
$$
and the corresponding curvature tensor is determined (up to the usual symmetries) by the non-zero components
$$
\begin{array}{l}
R_{1212}=\eta R_{1234}=\frac{1}{2} R_{3443}=2(\alpha^2+1),\\
\noalign{\medskip}
R_{1331}=R_{1324}=R_{1441}=R_{1432}=R_{2332}=R_{2442}=\eta(\alpha^2+1).
\end{array}
$$
It immediately follows that the Ricci operator is diagonal on the basis $\{ e_1,\dots,e_4\}$ of the form $\operatorname{Ric}=-6 (\alpha^2+1)\operatorname{diag}[0,0,1,1]$, and the scalar curvature  $\tau=-12 (\alpha^2+1)$.


\subsection*{Type II}
It follows from Theorem \ref{thm:characterization Type D} and Remark \ref{re:remark} that any
Type II generalized symmetric space is locally isometric to a metric Lie group $(G,g)$  whose Lie algebra is given  by $\mathfrak{g}=\operatorname{Span}\{ e_1,e_2,e_3,e_4\}$ with non-zero brackets
\[
\qquad
[e_1,e_2]=- 2\alpha e_3,\quad
[e_1,e_4]=\alpha e_1, \quad
[e_2,e_4]=-2 \alpha e_2,\quad
[e_3,e_4]=-\alpha e_3
\]
and the left-invariant metric which makes $\{ e_1,e_2,e_3,e_4\}$ an orthonormal basis of signature  $(--++)$.
Now, a straightforward calculation shows that the Levi-Civita connection of any Type II generalized symmetric space is given (with respect to a suitable orthonormal basis of left-invariant vector fields on $G$) by
$$
\begin{array}{llll}
\nabla_{e_1}e_1=\alpha e_4
&
\nabla_{e_1}e_2=-\alpha e_3
&
\nabla_{e_1}e_3=-\alpha e_2
&
\nabla_{e_1}e_4=\alpha e_1
\\
\noalign{\medskip}
\nabla_{e_2}e_1=\alpha e_3
&
\nabla_{e_2}e_2=-2\alpha e_4
&
\nabla_{e_2}e_3=\alpha e_1
&
\nabla_{e_2}e_4=-2\alpha e_2
\\
\noalign{\medskip}
\nabla_{e_3}e_1=-\alpha e_2
&
\nabla_{e_3}e_2=\alpha e_1
&
\nabla_{e_3}e_3=\alpha e_4
&
\nabla_{e_3}e_4=-\alpha e_3
\\
\noalign{\medskip}
\nabla_{e_4}e_1= 0
&
\nabla_{e_4}e_2= 0
&
\nabla_{e_4}e_3= 0
&
\nabla_{e_4}e_4=0
\end{array}
$$
and the corresponding curvature tensor is determined (up to the usual symmetries) by the non-zero components
$$
\begin{array}{l}
R_{1331}=R_{1342}=\frac{1}{2}R_{2424}=2\alpha^2,\\
\noalign{\medskip}
R_{1221}=R_{1243}=R_{1414}=R_{1432}=R_{2323}=R_{3443}=\alpha^2.
\end{array}
$$
It immediately follows that the Ricci operator is diagonal on the basis $\{ e_1,\dots,e_4\}$ of the form $\operatorname{Ric}=-6\alpha^2\operatorname{diag}[0,1,0,1]$, and the scalar curvature $\tau=-12\alpha^2$.

\subsection*{Type III}
As discussed in Section \ref{section:2.3}, consider the space $\mathbb{R}^4$ with coordinates $(x,y,u,v)$ and the metric
$$
g=\lambda(dx\circ dx+dy\circ dy+dx\circ  dy)+
e^{-y}(2dx+dy)\circ dv+e^{-x}(dx+2dy)\circ du,
$$
where $\lambda\in \mathbb{R}$. We analyze separately the conformally flat and non-conformally flat cases.
Assuming $\lambda=0$, let $\{ e_1,\dots,e_4\}$ be the  orthonomal frame given by
\[
\begin{array}{ll}
   e_1 =  \partial_y  - \frac{e^x}{2}\partial_u,
   &
   e_2 = \frac{2}{\sqrt{3}}\partial_x
   + \frac{7 e^x}{6 \sqrt{3}} \partial_u
   - \frac{4 e^y}{3\sqrt{3}}\partial_v,
   \\[0.2cm]
   e_3 = \partial_x
   -\frac{1}{2}\partial_y
   -\frac{e^x}{3} \partial_u
   + \frac{2 e^y}{3}\partial_v,
   \quad
   &
      e_4 = -\frac{1}{\sqrt{3}}\partial_x
   -\frac{\sqrt{3}}{2}\partial_y
   -\frac{4 e^x}{3 \sqrt{3}}\partial_u
   + \frac{2 e^y}{3 \sqrt{3}}\partial_v.
\end{array}
\]
Then the Levi-Civita connection is given by
$$
\begin{array}{ll}
\nabla_{e_1}e_1=
	\frac{1}{18} (2 \sqrt{3} e_2 - 15 e_3 + \sqrt{3} e_4)
	\quad
&
\nabla_{e_1}e_2=
	-\frac{1}{3\sqrt{3}}(e_1 + e_3 - \sqrt{3} e_4)
\\
\noalign{\medskip}
\nabla_{e_1}e_3=
	-\frac{5}{6}  e_1 - \frac{1}{3\sqrt{3}} (e_2 + 2 e_4)
&
\nabla_{e_1}e_4=
	\frac{1}{6\sqrt{3}} (e_1 + 2 \sqrt{3} e_2 + 4 e_3)

\\
\noalign{\medskip}

\nabla_{e_2}e_1=
	-\frac{1}{9} (5 e_2 - 2 (\sqrt{3} e_3 + e_4))
&
\nabla_{e_2}e_2=
	\frac{1}{18} (10 e_1 + 19 e_3 + 11 \sqrt{3} e_4)
\\
\noalign{\medskip}
\nabla_{e_2}e_3=
	\frac{1}{18} (4 \sqrt{3} e_1 + 19 e_2 + 20 e_4)
&
\nabla_{e_2}e_4=
	\frac{1}{18} (4 e_1 + 11 \sqrt{3} e_2 - 20 e_3)

\\
\noalign{\medskip}

\nabla_{e_3}e_1=
	\frac{1}{2} (\sqrt{3} e_2 - e_3 + \sqrt{3} e_4)
&
\nabla_{e_3}e_2=
	-\frac{1}{2\sqrt{3}} (3 e_1 - e_3 + \sqrt{3} e_4)
\\
\noalign{\medskip}
\nabla_{e_3}e_3=
	-\frac{1}{6} (3 e_1 - \sqrt{3} (e_2 - e_4))
&
\nabla_{e_3}e_4=
	\frac{1}{2\sqrt{3}} (3 e_1 - \sqrt{3} e_2 + e_3)
	
\\
\noalign{\medskip}

\nabla_{e_4}e_1=
	\frac{1}{18} (11 e_2 + \sqrt{3} e_3 + e_4)
&
\nabla_{e_4}e_2=
	-\frac{1}{18} (11 e_1 + 11 e_3 + 13 \sqrt{3} e_4)
\\
\noalign{\medskip}
\nabla_{e_4}e_3=
	\frac{1}{18} (\sqrt{3} e_1 - 11 e_2 - 13 e_4)
&
\nabla_{e_4}e_4=
	\frac{1}{18} (e_1 - 13 (\sqrt{3} e_2 - e_3))
\end{array}
$$
Further, the curvature tensor expresses as
$$
\begin{array}{l}
R_{1212} = -2 R_{1214} = 4R_{1414}= -4R_{2323}
= -2R_{2334} = - R_{3434} =\frac{14}{9},

\\[0.05in]
\noalign{\medskip}

  R_{1313} = -   R_{2424} = -\frac{1}{6},

\\[0.05in]
\noalign{\medskip}

R_{1213} = \frac{3}{4}R_{1224} = -2 R_{1314} = -\frac{3}{2}R_{1323}

\\[0.05in]
\phantom{R_{1213}}
= -\frac{3}{4}R_{1334} =
-\frac{3}{2}R_{1424} = -2R_{2324} =
-R_{2434} = \frac{1}{\sqrt{3}}.
\end{array}
$$

Next, assume $\lambda\neq0$ in \eqref{GENERtype B metric} and consider the orthonormal frame
\[
\begin{array}{ll}
   e_1 =  \frac{1}{\sqrt{\lambda}}\partial_y
   - e^x\sqrt{\lambda} \partial_u,
   \quad
   &
   e_2 = -\frac{2}{\sqrt{3 \lambda}} \partial_x
   + \frac{1}{\sqrt{3 \lambda}} \partial_y
    - \frac{ e^x \sqrt{\lambda}}{\sqrt{3}} \partial_u
    + \frac{2e^y \sqrt{\lambda}}{\sqrt{3}}\partial_v,
   \\[0.1cm]
   \noalign{\medskip}
   e_3 = \frac{1}{\sqrt{3 \lambda}} \partial_x
    -\frac{2}{\sqrt{3 \lambda}}\partial_y,
   &
   e_4 = \frac{1}{\sqrt{\lambda}}\partial_x.
\end{array}
\]
A straightforward calculation shows that the Levi-Civita connection is given by
$$
\begin{array}{ll}
\nabla_{e_1}e_1=
	-\frac{1}{3\sqrt{\lambda}} (\sqrt{3} e_2 + \sqrt{3} e_3 + 3 e_4)
	\quad
&
\nabla_{e_1}e_2=
	 \frac{1}{\sqrt{3\lambda}} e_1
\\
\noalign{\medskip}
\nabla_{e_1}e_3=
	-\frac{1}{\sqrt{3\lambda}} (e_1 - e_4)
&
\nabla_{e_1}e_4=
	-\frac{1}{3\sqrt{\lambda}} (3 e_1 + \sqrt{3} e_3)
\\
\noalign{\medskip}

\nabla_{e_2}e_1=
	-\frac{1}{3\sqrt{\lambda}} (3 e_2 + 3 e_3 + \sqrt{3} e_4)
&
\nabla_{e_2}e_2=
	\frac{1}{3\sqrt{\lambda}} (3 e_1 + 2 \sqrt{3} e_3)
\\
\noalign{\medskip}
\nabla_{e_2}e_3=
	-	\frac{1}{3\sqrt{\lambda}} (3 e_1 - 2 \sqrt{3} e_2 - 3 e_4)
&
\nabla_{e_2}e_4=
	-	\frac{1}{\sqrt{3\lambda}} (e_1 + \sqrt{3} e_3)

\\
\noalign{\medskip}

\nabla_{e_3}e_1=
		\frac{1}{6\sqrt{\lambda}}  (3 e_3 + \sqrt{3} e_4)
&
\nabla_{e_3}e_2=
	\frac{1}{6\sqrt{\lambda}} (\sqrt{3} e_3 - 3 e_4)
\\
\noalign{\medskip}
\nabla_{e_3}e_3=
	\frac{1}{6\sqrt{\lambda}} (3 e_1 + \sqrt{3} e_2)
&
\nabla_{e_3}e_4=
	\frac{1}{6\sqrt{\lambda}} (\sqrt{3} e_1 - 3 e_2)
	
\\
\noalign{\medskip}

\nabla_{e_4}e_1=
	\frac{1}{6\sqrt{\lambda}} (\sqrt{3} e_3 - 3 e_4)
&
\nabla_{e_4}e_2=
	- \frac{1}{6\sqrt{\lambda}} (3 e_3 + \sqrt{3} e_4)
\\
\noalign{\medskip}
\nabla_{e_4}e_3=
	\frac{1}{6\sqrt{\lambda}} (\sqrt{3} e_1 - 3 e_2)
&
\nabla_{e_4}e_4=
	-\frac{1}{6\sqrt{\lambda}} (3 e_1 + \sqrt{3} e_2)
\end{array}
$$
Hence, the curvature is determined  by
$$
\begin{array}{l}
\!\! R_{1212} = -3R_{1213} = -\sqrt{3} R_{1214}=
\sqrt{3}R_{1223}  = -3R_{1224}= 3R_{1234}

\\[0.05in]
\phantom{\!\!R_{1212}}
\!= 12 R_{1313} \!=4\sqrt{3}R_{1314}
\!= -4\sqrt{3}R_{1323} \!=
12 R_{1324}\!=4R_{1414}
\\[0.05in]
\phantom{\!\!R_{1212}}
 \!=-4R_{1423}
\!= 4\sqrt{3}R_{1424}\!=4R_{2323} \!= -4\sqrt{3}R_{2324} \!=
12R_{2424} \!= -3R_{3434} \!= \frac{2}{\lambda}.
\end{array}
$$

\begin{remark}\rm
An immediate consequence of the expressions of the Levi-Civita connection in this Appendix and the work in \cite{CalvarusoJGP} is that \emph{any four-dimensional generalized symmetric space is locally a Lie group}.

It follows from previous expressions for the Levi-Civita connection that any complete and simply connected Type III generalized symmetric space is a Lie group which is given by a Lie algebra as follows with respect to an orthonormal basis $\{ e_1,\dots, e_4\}$ of signature $(--++)$.
$$
\begin{array}{ll}
{[e_1,e_2]}\!=\!\frac{\sqrt{3}}{3\sqrt{\lambda}}\left( e_1+\sqrt{3}e_2+\sqrt{3}e_3+e_4\right)
&
{[e_1,e_3]}\!=\!-\frac{\sqrt{3}}{6\sqrt{\lambda}}\left( 2e_1-\sqrt{3}e_3+e_4  \right)
\\
\noalign{\medskip}
{[e_1,e_4]}\!=\!-\frac{1}{2\sqrt{\lambda}}\left( 2e_1-\sqrt{3} e_3+e_4 \right)
&
{[e_2,e_4]}\!=\!-\frac{\sqrt{3}}{6\sqrt{\lambda}}\left( 2e_1+\sqrt{3}e_3-e_4 \right)
\\
\noalign{\medskip}
{[e_2,e_3]}\!=\!\frac{1}{6\sqrt{\lambda}}\left( -6e_1+4\sqrt{3} e_2-\sqrt{3}e_3+9e_4 \right)
&
\end{array}
$$
if $\lambda\neq 0$, and
$$
\begin{array}{l}
{[e_1,e_2]}=\frac{1}{9}\left(-\sqrt{3}e_1+5e_2-3\sqrt{3}e_3+e_4 \right)
\\
\noalign{\medskip}
{[e_1,e_3]}=\frac{1}{18}\left(-15e_1-11\sqrt{3}e_2+9e_3-13\sqrt{3}e_4  \right)
\\
\noalign{\medskip}
{[e_1,e_4]}=\frac{1}{18}\left(\sqrt{3}e_1-5e_2+3\sqrt{3}e_3-e_4\right)
\\
\noalign{\medskip}
{[e_2,e_3]}=\frac{1}{18}\left(13\sqrt{3}e_1+19e_2-3\sqrt{3}e_3+29e_4\right)
\\
\noalign{\medskip}
{[e_2,e_4]}=\frac{1}{18}\left(15e_1+11\sqrt{3}e_2-9e_3+13\sqrt{3}e_4 \right)
\\
\noalign{\medskip}
{[e_3,e_4]}=\frac{1}{18}\left(8\sqrt{3}e_1+2e_2+3\sqrt{3}e_3+13e_4  \right)
\end{array}
$$
if $\lambda=0$.
\end{remark}

\section*{Acknowledgement}
The Authors would like to thank Prof. G. Calvaruso for his interest in our work and useful comments.



\begin{thebibliography}{99}
\bibitem{Afifi54}
Z. Afifi,
Riemann extensions of affine connected spaces,
\emph{Quart. J. Math., Oxford Ser. (2)} \textbf{5} (1954), 312--320.


\bibitem{Apost}
V. Apostolov,
Generalized Goldberg-Sachs theorems for
pseudo-Riemannian four-manifolds, \emph{J. Geom. Phys.} \textbf{27}
(1998), 185--198.

\bibitem{Apostolov-Armstrong-Draghici02}
V. Apostolov, J. Armstrong and T. Draghici,
Local rigidity of certain
classes of almost K\"{a}hler $4$-manifolds, \emph{Ann. Global Anal.
Geom.} \textbf{21} (2002), 151--176.

\bibitem{BK}
G. Bande and D. Kotschick,
The geometry of symplectic pairs,
\emph{Trans. Amer. Math. Soc.} \textbf{358} (2006), 1643--1655.

\bibitem{BO}
W. Batat and K. Onda,
Four-dimensional pseudo-Riemannian generalized symmetric spaces which are algebraic Ricci solitons,
\emph{Results Math.}, to appear.

\bibitem{BCGHM}
A. Bonome, R. Castro, E. Garc\'{\i}a-R\'{\i}o, L. Hervella and Y. Matsushita,
Almost complex manifolds with holomorphic distributions,
\emph{Rend. Mat. Appl. (7)}  \textbf{14}  (1994), 567--589.


\bibitem{Brozos-Garcia-Gilkey-Nikevic-Vazquez09}
M. Brozos-V\'{a}zquez, E. Garc\'{i}a-R\'{i}o, P. Gilkey, S. Nik\v{c}evi\'{c} and R. V\'{a}zquez-Lorenzo,
\emph{The geometry of Walker manifolds},
Synthesis Lectures on Mathematics and Statistics
\textbf{5}, Morgan \& Claypool Publ., 2009.

\bibitem{Brozos-Garcia-Gilkey-Vazquez07}
M. Brozos-V\'{a}zquez, E. Garc\'{\i}a-R\'{\i}o, P. Gilkey and R. V\'{a}zquez-Lorenzo,
Examples of signature $(2,2)$ manifolds with commuting curvature operators,
\emph{J. Phys. A} \textbf{40} (2007), 13149--13159.

\bibitem{Brozos-Garcia-Vazquez06}
M. Brozos-V\'{a}zquez, E. Garc\'{i}a-R\'{i}o and  R. V\'{a}zquez-Lorenzo,
Complete locally conformally flat manifolds of negative curvature,
\emph{Pacific J. Math.} \textbf{226} (2006), 201--219.

\bibitem{Calvaruso2}
G. Calvaruso,
Harmonicity of vector fields on four-dimensional generalized symmetric spaces,
\emph{Cent. Eur. J. Math.} \textbf{10} (2012), 411--425.

\bibitem{Calvaruso}
G. Calvaruso,
Symplectic, complex and K\"{a}hler structures on four-dimensional generalized symmetric spaces, \emph{Differential Geom. Appl.} \textbf{29} (2011), 758--769.

\bibitem{Calvaruso-Leo08}
G. Calvaruso and B. de Leo,
Curvature properties of four-dimensional
generalized symmetric spaces, \emph{J. Geom.} \textbf{90}  (2008),
30--46.

\bibitem{CZ}
G. Calvaruso and A. Zaeim,
Geometric Structures over Four-Dimensional Generalized Symmetric Spaces,
\emph{Mediterr. J. Math.} \textbf{10} (2013), 971--987.

\bibitem{CZ2}
G. Calvaruso and A. Zaeim,
Four-dimensional homogeneous Lorentzian manifolds, to appear.

\bibitem{CalvarusoJGP}
G. Calvaruso,
Addendum to: ``Homogeneous structures on three-dimensional Lorentzian manifolds'' [J. Geom. Phys. 57 (2007), no. 4, 1279–1291], \emph{J. Geom. Phys.} \textbf{58} (2008), 291--292.

\bibitem{CGRGVL09}
E. Calvi\~no-Louzao, E. Garc\'{\i}a-R\'{\i}o, P. Gilkey and R. V\'azquez-Lorenzo,
The geometry of modified Riemannian extensions,
\emph{Proc. R. Soc. Lond. Ser. A Math. Phys. Eng. Sci.} \textbf{465} (2009), 2023--2040.

\bibitem{Calvino-Garcia-Vazquez-Vazquez3}
E. Calvi\~{n}o-Louzao, E. Garc\'{\i}a-R\'{\i}o, M.E. V\'{a}zquez-Abal and R. V\'{a}zquez-Lorenzo,
Local rigidity and nonrigidity of symplectic pairs,
\emph{Ann. Glob. Anal. Geom.} \textbf{41} (2012), 241--252.


\bibitem{Calvino-Garcia-Vazquez10}
E. Calvi\~{n}o-Louzao, E. Garc\'{\i}a-R\'{\i}o and R. V\'{a}zquez-Lorenzo,
Riemann extensions of torsion-free connections with degenerate Ricci tensor,
\emph{Canad. J. Math.}, \textbf{62} (5), (2010), 1037--1057.

\bibitem{Cerny-Kowalski82}
J. $\check{C}$ern\'y and O. Kowalski,
Classification of generalized symmetric pseudo-Riemannian spaces of dimension $n\leq 4$,
\emph{Tensor, N.S} {\bf 38} (1982), 256--267.


\bibitem{Chen-Vanhecke}
B. Y. Chen and L. Vanhecke,
Isometric, holomorphic and symplectic reflections,
\emph{Geom. Dedicata} \textbf{29}  (1989), 259--277.

\bibitem{Chen-Vanhecke2}
B. Y. Chen and L. Vanhecke,
Symplectic reflections and complex space forms,
\emph{Ann. Glob. Anal. Geom.}  \textbf{9}  (1991),  205--210.

\bibitem{DAtri}
J. E. D'Atri,
Geodesic conformal transformations and symmetric spaces,
\emph{Kodai Math. Sem. Rep.} \textbf{26}  (1974/75), 201--203.

\bibitem{DN}
J. E. D'Atri and H. K. Nickerson,
Divergence-preserving geodesic symmetries,
\emph{J. Differential Geom.} \textbf{3} (1969), 467--476.

\bibitem{LV}
B. de Leo and J. Van der Veken,
Totally geodesic hypersurfaces of four-dimensional generalized symmetric spaces,
\emph{Geom. Dedicata} \textbf{159} (2012), 373--387.

\bibitem{Derdzinski78}
A. Derdzinski,
On homogeneous conformally symmetric pseudo-Riemannian manifolds,
\emph{Colloq. Math.}  \textbf{40} (1978/79),  167--185.

\bibitem{Derdzinski-Roter80}
A. Derdzinski and W. Roter,
Some properties of conformally symmetric manifolds which are not Ricci-recurrent,
\emph{Tensor (N.S.)} \textbf{34} (1980),  11--20.

\bibitem{Derdzinski-Roter97}
A. Derdzinski and W. Roter, Projectively flat surfaces, null parallel distributions, and
conformally symmetric manifolds, \emph{Tohoku Math. J.} \textbf{59} (2007), 565–-602.

\bibitem{DRGRVL}
J. C. D\'{\i}az-Ramos, E. Garc\'{\i}a-R\'{\i}o and R. V\'{a}zquez-Lorenzo,
Four-dimensional Osserman metrics with nondiagonalizable Jacobi operators,
\emph{J. Geom. Anal.} \textbf{16} (2006), 39--52.





\bibitem{Fino05}
A. Fino,
Almost K\"{a}hler $4$-dimensional Lie groups with $J$-invariant Ricci tensor,
\emph{Differential Geom. Appl.} \textbf{23} (2005), 26--37.






\bibitem{Garcia-Gilkey-Nikevic-Vazquez13}
E. Garc\'{i}a-R\'{i}o, P. Gilkey, S. Nik\v{c}evi\'{c} and R. V\'{a}zquez-Lorenzo,
\emph{Applications of affine and Weyl geometry},
Synthesis Lectures on Mathematics and Statistics
\textbf{13}, Morgan \& Claypool Publ., 2013.



\bibitem{EGR-Vanhecke}
E. Garc\'{\i}a-R\'{\i}o and L. Vanhecke,
Divergence-preserving geodesic transformations,
\emph{Proc. Roy. Soc. Edinburgh Sect. A} \textbf{128} (1998), 1309--1323.







\bibitem{Gray}
A. Gray,
Riemannian manifolds with geodesic symmetries of order 3,
\emph{J. Differential Geom.} \textbf{7} (1972), 343--369.

\bibitem{Hall10}
G. Hall,
Some remarks on the converse of Weyl's conformal theorem,
\emph{J. Geom. Phys.} \textbf{60} (2010), 1--7.

\bibitem{Ivanov-Zamkovoy05}
S. Ivanov and S. Zamkovoy,
Parahermitian and paraquaternionic manifolds,
\emph{Differential Geom. Appl.}  \textbf{23} (2005), 205--234.

\bibitem{Kowalski74}
O. Kowalski,
Riemannian manifolds with general symmetries,
\emph{Math. Z.} \textbf{136} (1974), 137--150.

\bibitem{Kowalski75}
O. Kowalski,
Classification of generalized symmetric Riemannian spaces of dimension $n\leq 5$,
Rozpravy \v{C}eskoslovensk\'e Akad.V\v{e}d \v{R}ada Mat.P\v{r}\'{\i}rod. \textbf{8} (1975), 61 pp.

\bibitem{Kowalski80}
O. Kowalski,
\emph{Generalized symmetric spaces},
Lecture Notes in Math., Springer-Verlag, Berlin, Heidelberg, New-York, \textbf{805},
1980.

\bibitem{KSek}
O. Kowalski and M. Sekizawa,
On natural Riemann extensions,
\emph{Publ. Math. Debrecen} \textbf{78} (2011), 709--721.

\bibitem{Ni-Va}
L. Nicolodi and L. Vanhecke,
Rotations and Hermitian symmetric spaces,
\emph{Monatsh. Math.} \textbf{109} (1990), 279--291.

\bibitem{Olszak90}
Z. Olszak,
On conformally recurrent manifolds. II. Riemann extensions,
\emph{Tensor, N.S.} \textbf{49} (1990), 24--31.


\bibitem{Patterson-Walker52}
E. M. Patterson and A. G. Walker, Riemann extensions, \emph{Quart. J.
Math., Oxford Ser. (2)} \textbf{3} (1952), 19--28.

\bibitem{Willmore-Vanhecke}
L. Vanhecke and T. J. Willmore,
Interaction of tubes and spheres,
\emph{Math. Ann.} \textbf{263} (1983), 31--42.


\end{thebibliography}
\end{document}